\documentclass[10pt]{article}
\usepackage{amsfonts}
\usepackage{amsmath,amsthm,amssymb,mathrsfs}
\usepackage{color}
\usepackage[hmargin=3cm,vmargin=3.5cm]{geometry}
\usepackage{graphicx}
\usepackage{cancel}

\newcommand{\bu}{\mathbf u}

\theoremstyle{plain}
\newtheorem{theorem}{Theorem}[section]
\newtheorem{lemma}[theorem]{Lemma}

\newtheorem{remark}{Remark}[section]
\numberwithin{equation}{section} \numberwithin{theorem}{section}
\numberwithin{remark}{section} 
\input{tcilatex}
\begin{document}

\title{Semigroup Well-posedness of A Linearized, Compressible Fluid with An
Elastic Boundary \thanks{%
The research of G. Avalos was partially supported by the NSF Grants
DMS-1211232 and DMS-1616425. The research of J.T. Webster was partially
supported by the NSF Grant DMS-1504697.}}
\author{George Avalos,\thanks{
University of Nebraska-Lincoln, gavalos@math.unl.edu} \and Pelin G. Geredeli,%
\thanks{%
Hacettepe University, Ankara, Turkey, and University of Nebraska-Lincoln,
pguven@hacettepe.edu.tr} \and Justin T. Webster, \thanks{%
University of Maryland, Baltimore County, Maryland, websterj@umbc.edu} }
\maketitle

\begin{abstract}
We address semigroup well-posedness of the fluid-structure interaction of a
linearized compressible, viscous fluid and an elastic plate (in the
absence of rotational inertia). Unlike existing work in the literature, we
linearize the compressible Navier-Stokes equations about an arbitrary state
(assuming the fluid is barotropic), and so the fluid PDE component of the
interaction will generally include a nontrivial ambient flow profile $%
\mathbf{U}$. The appearance of this term introduces new challenges at the
level of the stationary problem. In addition, the boundary of the fluid
domain is unavoidably Lipschitz, and so the well-posedness argument takes
into account the technical issues associated with obtaining necessary
boundary trace and elliptic regularity estimates. Much of the previous work
on flow-plate models was done via Galerkin-type constructions after
obtaining good a priori estimates on solutions (specifically \cite%
{Chu2013-comp}---the work most pertinent to ours here); in contrast, we
adopt here a Lumer-Phillips approach, with a view of associating solutions
of the fluid-structure dynamics with a $C_{0}$-semigroup $\left\{ e^{%
\mathcal{A}t}\right\} _{t\geq 0}$ on the natural finite energy space of
initial data. So, given this approach, the major challenge in our work becomes establishing of the maximality of the
operator $\mathcal{A}$ which models the fluid-structure dynamics. In sum:
our main result is semigroup well-posedness for the fully coupled
fluid-structure dynamics, under the assumption that the ambient flow field $%
\mathbf{U}\in \mathbf{H}^{3}(\mathcal{O})$ has zero normal component trace
on the boundary (a standard assumption with respect to the
literature). 
In the final sections we address
well-posedness of the system in the presence of the von Karman plate
nonlinearity, as well as the stationary problem associated with the
dynamics.


\end{abstract}

\begin{center}
\textit{{In memory of Igor D. Chueshov} }
\end{center}

\section{Introduction}

In this work, we consider a linearized fluid-structure model with respect to
some reference state, including an arbitrary spatial vector field. The
coupled system here describes the interaction between a plate and a
flow of \textit{compressible} barotropic, viscous fluid. Such interactive
dynamics are crucially considered in the design of many engineering systems
(e.g., aircraft, engines, and bridges). The study of compressible flows (gas
dynamics) itself has implications to high-speed aircraft, jet engines,
rocket motors, hyperloops, high-speed entry into a planetary atmosphere, gas
pipelines, commercial applications (such as abrasive blasting), and many
other fields (see \cite{BA62,bolotin,dowell1}, for instance). In these
applications, the density of a gas may change significantly along a
streamline. \emph{Compressibility}---i.e., the fractional change in volume
of the fluid element per unit change in pressure---becomes important, for
instance, in flows for which 
\begin{equation*}
\text{Mach Number}=M\equiv \frac{\text{velocity}}{\text{local speed of sound}%
}>0.3.
\end{equation*}%
The cases $M<0.3$ and $0.3<M<0.8$ are subsonic/incompressible and
subsonic/compressible regimes, respectively. Compressible flows can be
either transonic $(0.8<M<1.2)$ or supersonic $(1.2<M<3.0)$. In supersonic
flows, pressure effects are only transported downstream; the upstream flow
is not affected by conditions downstream.

In the study of incompressible flows, the associated analysis typically
involves only two unknowns: pressure and velocity. These are usually found
by solving two equations that describe conservation of mass and linear
momentum, with the fluid density presumed to be constant. By contrast, in
compressible flow, the gas density and temperature are variables.
Consequently, the solution of compressible flow problems will require two
more equations: namely, an equation of state for the gas, and a conservation
of energy equation.\footnote{%
Throughout, we will assume the fluid is barotropic---the pressure depends
only on the density.} Moreover, the imposition of external forces to the
governing equations may not immediately result in a uniform flow throughout
the system. In particular, the fluid may compress in the vicinity of the
applied force; that is to say, the density may increase locally in response
to the given force.

The effects due to compressibility and viscosity on an (uncoupled) fluid
dynamics will have to be taken in account when subsequently considering the
mathematical properties of PDE's describing interactions of said fluid
dynamics with some given structure. In aeroelasticity, the compressible gas
is often assumed to be inviscid---i.e., viscosity-free---and the flow
irrotational (potential flow). These assumptions are often invoked in
practice, as they reduce the flow dynamics to a wave equation \cite{dcds,
dowell1} (and see Section \ref{techreview} below). However, there are situations
where viscous effects cannot be neglected, e.g., in the transonic region 
\cite{dowell1}. The \emph{mathematical} literature---especially in the last
20 years---on fluid-structure interactions across each of these fluid
regimes is quite vast. We will certainly not attempt here a general overview
of this literature, but in Section \ref{techreview} we will provide an
in-depth discussion of those key modern references that pertain to the present work. At this point, we mention the primary motivating reference, 
\cite{Chu2013-comp}: in this work, the author considers the dynamics of a
nonlinear plate, located on a flat portion of the boundary of a three
dimensional cavity, as it interacts with a compressible, barotropic
(linearized) fluid that fills the cavity. In the present work, we will
analyze a comparable setup, but with additional physical terms in the equations; the focus here
will be on establishing and describing the essential semigroup dynamics
which drives the coupled PDE model.

\begin{remark}
The accommodation of physically relevant nonlinearities---i.e., those seen
in \cite{Chu2013-comp, cr-full-karman, supersonic}---can be readily made
subsequent to the present analysis, which develops a \textquotedblleft
good\textquotedblright\ linear theory. Linearities that are amenable to such
treatment include those of Berger, Kirchhoff, or von Karman type, inasmuch
as they are \emph{locally Lipschitz} \cite{springer,pazy} on the plate's
natural energy space. As a key illustrative example, we include a discussion of the
well-posedness of this fluid-structure model in the presence of the von
Karman plate nonlinearity in Section \ref{nonlinear}.

In many cases, it is the addition of structural nonlinearity which
ultimately leads to global-in-time boundedness of corresponding trajectories 
\cite{conequil2,webster}. Accordingly, long-time behavior of nonlinear
dynamics will be considered in a forthcoming work.
\end{remark}

We also consider a Lipschitz geometry, as opposed to the common assumption 
\cite{Chu2013-comp,dV} that the domain is smooth. Given the transition
between the elastic and inelastic components of the boundary, a Lipschitz
boundary is surely more natural and physically relevant---see Figure 1
below---and also more amenable to pertinent generalizations, e.g., tubular
domains (finite or infinite) \cite{tube,ChuRyz2012-pois}. Distinguishing our work from \cite{Chu2013-comp}, we take the linearization
of the compressible Navier-Stokes equations about a rest state which has a 
\emph{nonzero ambient flow component}. Since this linearization process
produces some additional terms that depend on the ambient flow, previous
techniques to obtain the well-posedness result cannot be directly applied%
\footnote{%
In fact, the late author of \cite{Chu2013-comp}---to whom this work is
dedicated---suggested in personal correspondence the precise model (\ref%
{system1})--(\ref{IC_2}); \cite{Igor-note}. In this communication, he
remarked that the approaches utilized in \cite{Chu2013-comp} with $\mathbf{U}%
\equiv 0$ were not amenable to the problem studied with $\mathbf{U}\neq 
\mathbf{0}$. He noted that a semigroup approach might be fruitful, and this
comment provided an impetus for the present work.}. We note that the
resultant terms, due to the presence of the ambient flow, \textbf{do not represent
bounded perturbations of principal spatial operators.}  To obtain
the primary result we utilize a semigroup approach, invoking the well known
Lumer-Phillips theorem \cite[p.13]{pazy}.

We believe that the present treatment fits nicely within the context of the
recent work of I. Chueshov, where the interactive dynamics between fluid and
a plate (or shell) are considered from various points of view \cite%
{Chu2013-comp,Chu2013-inviscid,tube,cr-full-karman,ChuRyz2012-pois,ChuRyz2011}%
. Moreover, one can draw comparisons and contrasts between the
well-posedness analysis here and that in \cite{clark} and \cite{george1},
which deal with \emph{incompressible }fluid-structure interactions: in \cite%
{clark} and \cite{george1}, there is also a two dimensional elastic
structure existing on the boundary of a three dimensional domain, in which a
fluid evolves. However, the earlier well-posedness work \cite{clark}
requires an appropriate \emph{mixed variational}, \emph{Babu\u{s}ka-Brezzi}
formulation, which is nonstandard and instrinsic to the particular dynamics
under consideration (see e.g., Theorem 3.1.5 of \cite{kesavan}); whereas the
present effort combines the Lax-Milgram Theorem with a critical
well-posedness result for (static versions of) the pressure PDE component of
the fluid-structure system (the first equation in (\ref{system1}) and
Theorem \ref{dV} of the Appendix.) For fluid-structure well-posedness
studies that involve a three dimensional solid immersed in a three
dimensional fluid, and which utilize semigroup techniques, see \cite{T1},%
\cite{T2},\cite{dvorak}.)

Eventually, we are interested in learning if and how the presence of the
dissipating fluid dynamics affects the stability of the structure (as in,
e.g., \cite{delay,Chu2013-comp,ChuRyz2011,ChuRyz2012-pois},\cite{george2}).
In particular, for the \emph{linear} compressible fluid-structure
interaction PDE model, we are interested in strong/uniform stability
properties of the associated $C_{0}$-semigroup. On the other hand, if one
inserts nonlinearity into the structural component of the interaction, the
existence and nature of global attractors become the primary objects of
interest for the associated PDE dynamical system. Qualitative properties of
fluid-structure models (such as well-posedness and stability of solutions,
and the existence of compact global attractors) have been intensely
investigated by many authors over the past 30 years. For the PDE model under
consideration, (\ref{system1})--(\ref{IC_2}) below, issues of long-time
behavior of solutions are addressed in the forthcoming work \cite{preprint}.

The paper is organized as follows: In Section \ref{model} we describe the
PDE model and discuss our standing hypotheses. In Section \ref{results} we
provide a discussion of the principal dynamics operator (on the natural
space of finite energy), as well its domain; we then formally state the
semigroup generation result which immediately yields well-posedness of
fluid-structure model, in the sense of Hadamard. We also include in this
section the notion of \emph{energy balance} for semigroup solutions. Section %
\ref{techreview} provides an in-depth discussion of the key pertinent
references, and their relationship to the result presented here. Section \ref%
{proof} gives the proof of the main result via the Lumer-Phillips theorem:
namely, we establish dissipativity and maximality of a certain bounded
perturbation of the modeling fluid-structure operator. Section \ref{static}
gives a description of stationary solutions to the dynamics at hand, and
proves that the PDE system can be recovered (in some sense) from the
stationary variational problem. Section \ref%
{nonlinear} discusses the von Karman plate nonlinearities and the
associated nonlinear dynamic well-posedness and stationary results, treating the nonlinearity as a
locally-Lipschitz perturbation of the linear dynamics.  Lastly, the Appendix provides a proof of a
key technical lemma on the well-posedness of solutions to the stationary
version of the decoupled pressure PDE component in (\ref{system1})--(\ref%
{IC_2}) below.

\subsection{Notation}

For the remainder of the text we write $\mathbf{x}$ for $(x_{1},x_{2},x_{3})%
\in \mathbb{R}_{+}^{3}$ or $(x_{1},x_{2})\in \Omega \subset \mathbb{R}%
_{\{(x_{1},x_{2})\}}^{2}$, as dictated by context. For a given domain $D$,
its associated $L^{2}(D)$ will be denoted as $||\cdot ||_D$ (or simply $%
||\cdot||$ when the context is clear). The symbols $\mathbf{n}$ and $%
\boldsymbol{\tau }$ will be used to denote, respectively, the unit external
normal and tangent vectors to $\mathcal{O}$. Inner products in $L^{2}(%
\mathcal{O})$ or $\mathbf{L}^{2}(\mathcal{O})$ are written $(\cdot ,\cdot)_{%
\mathcal{O}}$ (or simply $(\cdot ,\cdot)$ when the context is clear), while
inner products $L^{2}(\partial \mathcal{O})$ are written $\langle \cdot
,\cdot \rangle$. We will also denote pertinent duality pairings as $%
\left\langle \cdot ,\cdot \right\rangle _{X\times X^{\prime }}$, for a given
Hilbert space $X$. The space $H^{s}(D)$ will denote the Sobolev space of
order $s$, defined on a domain $D$, and $H_{0}^{s}(D)$ denotes the closure
of $C_{0}^{\infty }(D)$ in the $H^{s}(D)$-norm $\Vert \cdot \Vert
_{H^{s}(D)} $ or $\Vert \cdot \Vert _{s,D}$. We make use of the standard
notation for the boundary trace of functions defined on $\mathcal{O}$, which
are sufficently smooth: i.e., for a scalar function $\phi \in H^{s}(\mathcal{%
O})$, $\frac{1}{2}<s<\frac{3}{2}$, $\gamma (\phi )=\phi \big|_{\partial 
\mathcal{O}},$ a well-defined and surjective mapping on this range of $s$,
owing to the Sobolev Trace Theorem on Lipschitz domains (see e.g., \cite%
{necas}, or Theorem 3.38 of \cite{Mc}).

\section{PDE Model}

\label{model}

Let $\mathcal{O}\subset \mathbb{R}^{3}$ be a \emph{bounded} and \emph{convex 
}fluid domain (and so has Lipschitz boundary $\partial \mathcal{O}$; see
e.g., Corollary 1.2.2.3 of \cite{grisvard}). The boundary decomposes into
two pieces $\overline{S}$ and $\overline{\Omega }$ where $\partial \mathcal{O%
}=\overline{S}\cup \overline{\Omega }$, with $S\cap \Omega =\emptyset $. We
consider $S$ to be the solid boundary, with no interactive dynamics, and $%
\Omega $ to be the equilibrium position of the elastic domain, upon which
the interactive dynamics takes place. We also assume that: (i) the active
component $\Omega \subset \mathbb{R}^{2}$ is flat, w\i th L\i psch\i tz
boundary, and embedded in the $x_{1}-x_{2}$ plane; (ii) the inactive
component $S$ l\i es below the $x_{1}-x_{2}$ plane. This is to say, 
\begin{align}
\Omega \subset & ~\{\mathbf{x}=(x_{1},x_{2},0)\} \\
S\subset & ~\{\mathbf{x}=(x_{1},x_{2},x_{3})~:~x_{3}\leq 0\}.
\end{align}%
Letting $\mathbf{n}(\mathbf{x})$ denote the unit outward normal vector to $%
\partial \mathcal{O}$, we have $\left. \mathbf{n}\right\vert _{\Omega
}=(0,0,1).$ (See Figure 1.)

\begin{figure}[tbp]
\begin{center}
\includegraphics[scale=.8]{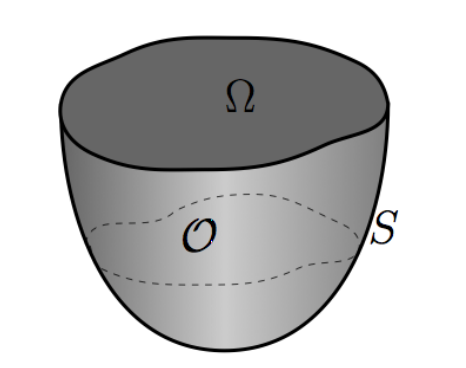}
\end{center}
\caption{The Fluid-Structure Geometry}
\end{figure}


We consider the compressible Navier-Stokes system \cite{chorin-marsden},
assuming the fluid is barotropic, and linearize the system with respect to
some reference rest state of the form $\left\{ p_{\ast },\mathbf{U},\varrho
_{\ast }\right\} $. The pressure and density components ${p_{\ast },\varrho
_{\ast }}$ are assumed to be scalar constants, and the arbitrary ambient
flow field $\mathbf{U}:\mathcal{O}\rightarrow \mathbb{R}^{3}$ is given by: 
\begin{equation}
\mathbf{U}%
(x_{1},x_{2},x_{3})=[U_{1}(x_{1},x_{2},x_{3}),U_{2}(x_{1},x_{2},x_{3}),U_{3}(x_{1},x_{2},x_{3})].
\label{flowfield}
\end{equation}%
Deleting non-critical lower order terms (see Remark \ref{delete} below), and
setting the pressure and density reference constants equal to unity, we
obtain the following \emph{perturbation equations}: 
\begin{align}
& \left\{ 
\begin{array}{l}
p_{t}+\mathbf{U}\cdot \nabla p+div~\mathbf{u}=0~\text{ in }~\mathcal{O}%
\times (0,\infty ) \\ 
\mathbf{u}_{t}+\mathbf{U}\cdot \nabla \mathbf{u}-div~\sigma (\mathbf{u}%
)+\eta \mathbf{u}+\nabla p=0~\text{ in }~\mathcal{O}\times (0,\infty ) \\ 
(\sigma (\mathbf{u})\mathbf{n}-p\mathbf{n})\cdot \boldsymbol{\tau }=0~\text{
on }~\partial \mathcal{O}\times (0,\infty ) \\ 
\mathbf{u}\cdot \mathbf{n}=0~\text{ on }~S\times (0,\infty ) \\ 
\mathbf{u}\cdot \mathbf{n}=w_{t}~\text{ on }~\Omega \times (0,\infty )%
\end{array}%
\right.  \label{system1} \\
&  \notag \\
& \left\{ 
\begin{array}{l}
w_{tt}+\Delta ^{2}w+\left[ 2\nu \partial _{x_{3}}(\mathbf{u})_{3}+\lambda 
\text{div}(\mathbf{u})-p\right] _{\Omega }=0~\text{ on }~\Omega \times
(0,\infty ) \\ 
w=\frac{\partial w}{\partial \nu }=0~\text{ on }~\partial \Omega \times
(0,\infty )%
\end{array}%
\right.  \label{IM2} \\
&  \notag \\
& 
\begin{array}{c}
\left[ p(0),\mathbf{u}(0),w(0),w_{t}(0)\right] =\left[ p_{0},\mathbf{u}%
_{0},w_{0},w_{1}\right] .%
\end{array}
\label{IC_2}
\end{align}%
Here, $p(t):\mathbb{R}^{3}\rightarrow \mathbb{R}$ and $\mathbf{u}(t):\mathbb{%
R}^{3}\rightarrow \mathbb{R}^{3}$ (pointwise in time) are given as the
pressure and the fluid velocity field, respectively. The quantity $\eta >0$
represents a drag force of the domain on the viscous fluid. In addition, the
quantity $\mathbf{\tau }$ in (\ref{system1}) is in the space $%
TH^{1/2}(\partial \mathcal{O)}$ of tangential vector fields of Sobolev index
1/2; that is,%
\begin{equation}
\mathbf{\tau }\in TH^{1/2}(\partial \mathcal{O)=}\{\mathbf{v}\in \mathbf{H}^{%
\frac{1}{2}}(\partial \mathcal{O})~:~\mathbf{v}\cdot \mathbf{n}=0~\text{ on }%
~\partial \mathcal{O}\}.\footnote{%
See e.g., p.846 of \cite{buffa2}.}  \label{TH}
\end{equation}

With respect to the \textquotedblleft ambient flow\textquotedblright\ field $%
\mathbf{U}$ appearing in (\ref{system1}), we define the space 
\begin{equation}
\mathbf{V}_{0}=\{\mathbf{v}\in \mathbf{H}^{1}(\mathcal{O})~:~\left. \mathbf{v%
}\right\vert _{\partial \mathcal{O}}\cdot \mathbf{n}=0~\text{ on }~\partial 
\mathcal{O}\};  \label{V_0}
\end{equation}%
and subsequently impose the standard assumption that 
\begin{equation}
\mathbf{U}\in \mathbf{V}_{0}\cap \mathbf{H}^{3}(\mathcal{O})  \label{min}
\end{equation}%
(see the analogous---and actually slightly stronger---specifications made on
ambient fields on p.529 of \cite{dV} and pp.102--103 of \cite{valli}).

\begin{remark}
As mentioned above, the presence of $\mathbf{U}$ in the modeling introduces
the term $\mathbf{U}\cdot \nabla p$ into the pressure equation, which {\bf
does not} represent a bounded perturbation of the dynamics.
\end{remark}

Given the \textit{Lam\'{e} Coefficients }$\lambda \geq 0$ and $\nu >0$, the 
\textit{stress tensor} $\sigma $ of the fluid is defined as 
\begin{equation*}
\sigma (\mathbf{\mu })=2\nu \epsilon (\mathbf{\mu })+\lambda \lbrack
I_{3}\cdot \epsilon (\mathbf{\mu })]I_{3},
\end{equation*}%
where the \textit{strain tensor }$\epsilon $ is given by 
\begin{equation*}
\epsilon _{ij}(\mu )=\dfrac{1}{2}\left( \frac{\partial \mu _{j}}{\partial
x_{i}}+\frac{\partial \mu _{i}}{\partial x_{j}}\right) \text{, \ }1\leq
i,j\leq 3
\end{equation*}%
(see \cite[p.129]{kesavan}). With this notation it is easy to see that 
\begin{equation*}
\text{div}~\sigma (\mathbf{\mu })=\nu \Delta \mathbf{\mu }+(\nu +\lambda
)\nabla \text{div}(\mathbf{\mu }),
\end{equation*}%
where $\lambda $ and $\nu $ are the non-negative viscosity coefficients.

The boundary conditions that are invoked in (\ref{system1}) for the fluid
PDE component are the so-called \emph{impermeability}-slip conditions \cite%
{bolotin,chorin-marsden}. Their physical interpretation is that no fluid
passes through the boundary (the normal component of the fluid field $%
\mathbf{u}$ on the active boundary portion $\Omega $ matches the plate
velocity $w_{t}$), and that there is no stress in the tangential direction $%
\tau $.
\begin{remark}
Other possible physically relevant boundary conditions have appeared in the
literature. We mention the \emph{Kutta-Joukowski} type condition \eqref{KJ} 
\cite{dcds}, as well as the \emph{adherence} condition \eqref{ad}. 
\begin{align}
\sigma (\mathbf{u})\mathbf{n}-p\mathbf{n}=\mathbf{0}~\text{ on }~S;& ~~~%
\mathbf{u}\cdot \mathbf{n}=w_{t}~\text{ on }~\Omega ;  \label{KJ} \\
\mathbf{u}=0~\text{ on }~S& ~~~\mathbf{u}\cdot \mathbf{n}=w_{t}~\text{ on }%
~\Omega.  \label{ad}
\end{align}
\end{remark}

Though the focus of this treatment is on the {\em linear dynamics} of the fluid-plate interaction, we do provide a brief discussion of nonlinearity in the model in Section \ref{nonlinear}. We now mention the principal nonlinear plate model of interest: the scalar von Karman plate. Writing the plate equation in \eqref{IM2} as 
\begin{equation}\label{IM2**}
w_{tt}+\Delta^2w+\left[ 2\nu \partial _{x_{3}}(\mathbf{u})_{3}+\lambda 
\text{div}(\mathbf{u})-p\right] _{\Omega }=f(w)~\text{ on }~\Omega \times
(0,\infty )
\end{equation} where we have
 $$f(w)=[w, v(w)+F_0],$$ where $F_0$ is a given function
  from $H^4(\Omega)$ and
the von Karman bracket $[u,v]$  is given by
\begin{equation*}\label{bracket}
[u,w] = \partial_{xx} u\cdot \partial_{yy} w +
\partial_{yy} u\cdot \partial_{xx} w -
2\cdot \partial_{xy} u\cdot \partial_{xy}
w,
\end{equation*} and
the Airy stress function $v(u,w) $ solves the following  elliptic
problem
\begin{equation}\label{airy-1}
\Delta^2 v(u,w)+[u,w] =0 ~~{\text in}~~  \Omega,\hskip.5cm \partial_{\nu} v(u,w) = v(u,w) =0 ~~{\text on}~~  \partial\Omega.
\end{equation}
Von  Karman equations are well known in nonlinear elasticity and
constitute a basic model describing nonlinear oscillations of a
plate accounting for  large deflections, see \cite{springer} and
references therein.
\begin{remark} In this paper we provide a discussion of the most physically relevant {\em large deflection} plate model. We do not fully discuss the breadth of nonlinear plate dynamics, as is done in \cite{Chu2013-comp}. However, the discussion we provide here is easily adapted to the other common plate nonlinearities of Berger or Kirchhoff type (see, for instance, \cite{supersonic} and \cite{Chueshov,gw}). 
\end{remark}

\begin{remark}
\label{delete} The above fluid equations in \eqref{system1} might be
referred to as the Oseen equations for viscous compressible barotropic
fluids. In the linearization procedure, without making additional
assumptions on $\mathbf{U}$, we obtain: 
\begin{align} \label{bar-model1-U}
& (\partial _{t}+\mathbf{U}\cdot \nabla )p+\mathrm{div\,}\,\mathbf{u}+%
\mathrm{div}(\mathbf{U})p=F_{1}(\mathbf{x})\quad \mathrm{in}~~\mathcal{O}%
\times \mathbb{R}_{+},   \\[2mm]
& (\partial _{t}+\mathbf{U}\cdot \nabla )\mathbf{u}-\nu \Delta \mathbf{u}%
-(\nu +{\lambda }){\nabla }\mathrm{div\,}\,\mathbf{u}+{\nabla }p+\nabla 
\mathbf{U}\cdot \mathbf{u}+(\mathbf{U}\cdot \nabla \mathbf{U})p=\mathbf{F}%
_{2}(\mathbf{x})\quad \mathrm{in}~\mathcal{O}\times \mathbb{R}_{+},
\label{flu-eq1U}
\end{align}
for a prescribed scalar function $F_{1}$ and vector field $\mathbf{F}_{2}$.
In our analysis we retain only the principal mathematical terms in %
\eqref{system1}--\eqref{IC_2}, as the others may be viewed as zeroth order
perturbations, and handled in a standard fashion.
\end{remark}

\section{Main Results}

\label{results} We are primarily nterested in Hadamard well-posedness of the
linearized coupled system given in (\ref{system1})--(\ref{IC_2}).
Specifically, we will ascertain well-posedness of the PDE model (\ref%
{system1})--(\ref{IC_2}) for arbitrary initial data in the natural space of
finite energy. To accomplish this, we will adopt a semigroup approach;
namely, we will pose and validate an explicit semigroup generator
representation for the fluid-structure dynamics (\ref{system1})--(\ref{IC_2}%
).

With respect to the coupled PDE system (\ref{system1})--(\ref{IC_2}), the
associated space of well-posedness will be 
\begin{equation}
\mathcal{H}\equiv L^{2}(\mathcal{O})\times \mathbf{L}^{2}(\mathcal{O})\times
H_{0}^{2}(\Omega )\times L^{2}(\Omega ).  \label{H}
\end{equation}%
$\mathcal{H}$ is a Hilbert space, topologized by the following
inner-product: 
\begin{equation}
(\mathbf{y}_{1},\mathbf{y}_{2})_{\mathcal{H}}=(p_{1},p_{2})_{L^{2}(\mathcal{O%
})}+(\mathbf{u}_{1},\mathbf{u}_{2})_{\mathbf{L}^{2}(\mathcal{O})}+(\Delta
w_{1},\Delta w_{2})_{L^{2}(\Omega )}+(v_{1},v_{2})_{L^{2}(\Omega )}
\label{innerp}
\end{equation}%
for any $\mathbf{y}_{i}=(p_{i},\mathbf{u}_{i},w_{i},v_{i})\in \mathcal{H}%
,~i=1,2.$

In what follows, we consider the linear operator $\mathcal{A}:D(\mathcal{A}%
)\subset \mathcal{H}\rightarrow \mathcal{H}$, which expresses the
compressible fluid-structure PDE system (\ref{system1})--(\ref{IC_2}) as the
abstract ODE: 
\begin{eqnarray}
\dfrac{d}{dt}%
\begin{bmatrix}
p \\ 
\mathbf{u} \\ 
w \\ 
w_{t}%
\end{bmatrix}
&=&\mathcal{A}%
\begin{bmatrix}
p \\ 
\mathbf{u} \\ 
w \\ 
w_{t}%
\end{bmatrix}%
;  \notag \\
\lbrack p(0),\mathbf{u}(0),w(0),w_{t}(0)] &=&[p_{0},\mathbf{u}%
_{0},w_{0},w_{1}].  \label{ODE}
\end{eqnarray}

\noindent To wit, 
\begin{equation}
\mathcal{A}=\left[ 
\begin{array}{cccc}
-\mathbf{U}\mathbb{\cdot }\nabla (\cdot ) & -\func{div}(\cdot ) & 0 & 0 \\ 
-\mathbb{\nabla (\cdot )} & \func{div}\sigma (\cdot )-\eta I-\mathbf{U}%
\mathbb{\cdot \nabla (\cdot )} & 0 & 0 \\ 
0 & 0 & 0 & I \\ 
\left. \left[ \cdot \right] \right\vert _{\Omega } & -\left[ 2\nu \partial
_{x_{3}}(\cdot )_{3}+\lambda \func{div}(\cdot )\right] _{\Omega } & -\Delta
^{2} & 0%
\end{array}%
\right] .  \label{AAA}
\end{equation}

\noindent Here, the domain $D(\mathcal{A})$ is given as 
\begin{equation*}
D(\mathcal{A})=\{(p_{0},\mathbf{u}_{0},w_{1},w_{2})\in L^{2}(\mathcal{O}%
)\times \mathbf{H}^{1}(\mathcal{O})\times H_{0}^{2}(\Omega )\times
H_{0}^{2}(\Omega )~:~(i)\text{--}(v)~~\text{hold below}\},
\end{equation*}%
where

\begin{enumerate}
\item[(A.i)] $\mathbf{U}\cdot \nabla p_{0}\in L^{2}(\mathcal{O})$

\item[(A.ii)] $\text{div}~\sigma (\mathbf{u}_{0})-\nabla p_{0}\in L^{2}(%
\mathcal{O})$

\item[(A.iii)] $-\Delta ^{2}w_{0}-\left[ 2\nu \partial _{x_{3}}(\mathbf{u}%
_{0})_{3}+\lambda \text{div}(\mathbf{u}_{0})\right] _{\Omega }+\left.
p_{0}\right\vert _{\Omega }\in L^{2}(\Omega )$

\item[(A.iv)] $\left( \sigma (\mathbf{u}_{0})\mathbf{n}-p_{0}\mathbf{n}%
\right) \bot ~TH^{1/2}(\partial \mathcal{O})$. That is, 
\begin{equation*}
\left\langle \sigma (\mathbf{u}_{0})\mathbf{n}-p_{0}\mathbf{n},\mathbf{\tau }%
\right\rangle _{\mathbf{H}^{-\frac{1}{2}}(\partial \mathcal{O})\times 
\mathbf{H}^{\frac{1}{2}}(\partial \mathcal{O})}=0\text{ \ for all }\mathbf{%
\tau }\in TH^{1/2}(\partial \mathcal{O}).
\end{equation*}

\item[(A.v)] $\mathbf{u}_{0}=\mathbf{\mu }_{0}+\widetilde{\mathbf{\mu }}_{0}$%
, where $\mathbf{\mu }_{0}\in \mathbf{V}_{0}$ and $\widetilde{\mathbf{\mu }}%
_{0}\in \mathbf{H}^{1}(\mathcal{O})$ satisfies\footnote{%
The existence of an $\mathbf{H}^{1}(\mathcal{O})$-function $\widetilde{%
\mathbf{\mu }}_{0}$ with such a boundary trace on Lipschitz domain $\mathcal{%
O}$ is assured; see e.g., Theorem 3.33 of \cite{Mc}, or see also the proof
of Lemma \ref{staticwellp} below.}%
\begin{equation*}
\left. \widetilde{\mathbf{\mu }}_{0}\right\vert _{\partial \mathcal{O}}=%
\begin{cases}
0 & ~\text{ on }~S \\ 
w_{2}\mathbf{n} & ~\text{ on}~\Omega%
\end{cases}%
\end{equation*}%
\noindent (and so $\left. \mathbf{\mu }_{0}\right\vert _{\partial \mathcal{O}%
}\in TH^{1/2}(\partial \mathcal{O})$).
\end{enumerate}

\bigskip

In the following theorem, we provide semigroup well-posedness for $\mathcal{A%
}:D(\mathcal{A})\in \mathcal{H}\rightarrow \mathcal{H}$, the proof of which
is based on the well known Lumer-Phillips Theorem.

\begin{theorem}
\label{wellp} The map $\left\{ p_{0},\mathbf{u}_{0};w_{0},w_{1}\right\}
\rightarrow \left\{ p(t),\mathbf{u}(t);w(t),w_{t}(t)\right\} $ defines a
strongly continuous semigroup $\{e^{\mathcal{A}t}\}$ on the space $\mathcal{H%
}$, and hence the system (\ref{system1})--(\ref{IC_2}) is well-posed (in the
sense of mild solutions---see Remark \ref{notion} below). In addition, the
semigroup enjoys the following estimate: 
\begin{equation}
\big|\big|e^{\mathcal{A}t}\big|\big|_{\mathcal{L}(\mathcal{H})}\leq \exp %
\Big(\dfrac{t}{2}||\text{div}(\mathbf{U})||_{\infty }\Big),~~\forall t>0.
\end{equation}
\end{theorem}

\begin{remark}
\label{weaker}Given the existence of a semigroup $\{e^{\mathcal{A}t}\}$ for
the fluid-structure generator $\mathcal{A}:D(\mathcal{A})\subset \mathcal{H}%
\rightarrow \mathcal{H}$: if initial data $[p_{0},\mathbf{u}%
_{0};w_{0},w_{1}]\in D(\mathcal{A})$, the corresponding solution $[p(t),%
\mathbf{u}(t);w(t),w_{t}(t)]\in C([0,\infty ),D(\mathcal{A}))$. In
particular, the solution satisfies the condition (A.iv) in the definition of
the generator. This means that one has that the tangential boundary
condition 
\begin{equation*}
\lbrack \sigma (\mathbf{u}_{0})\mathbf{n}-p_{0}\mathbf{n}]\cdot \mathbf{\tau 
}=0\text{\ \ for all }\mathbf{\tau }\in TH^{1/2}(\partial \mathcal{O}),
\end{equation*}%
satisfied in the sense of distributions. That is to say, $\forall ~
\boldsymbol{\tau }\in TH^{1/2}(\partial \mathcal{O})$ and $\forall \phi \in 
\mathcal{D}(\partial \mathcal{O})$,%
\begin{equation}
\langle \sigma (\mathbf{u}_{0})\mathbf{n}-p_{0}\mathbf{n},\phi \boldsymbol{%
\tau }\rangle _{\partial \mathcal{O}}=0.  \label{weak2}
\end{equation}
\end{remark}

\begin{remark}[Notions of Solution]
\label{notion} We note that semigroup solutions, as arrived at in Theorem %
\ref{wellp} for initial data $\mathbf{y}\in \mathcal{H}$, correspond to so
called \emph{mild} solutions (satisfying an integral form of \eqref{system1}%
--\eqref{IC_2}) in the sense of \cite[Section 4.2]{pazy}. Moreover, for
initial data $\mathbf{y}\in D(\mathcal{A})$, we obtain so called \emph{%
strong solutions}, which satisfy the PDE in a pointwise sense.

In \cite{Chu2013-comp}, semigroup techniques are not used in demonstrating
well-posedness. As such, the author takes care to define an appropriate
notion of \emph{weak solution} corresponding to a Galerkin construction (see 
\cite[pp.653--654]{Chu2013-comp}). Such a notion of weak solution is
relevant here, and can be obtained by making minor modifications that take
into account the vector field $\mathbf{U}$. Here we assert that mild
solutions (obtained via our semigroup) are in fact weak solutions as in \cite%
{Chu2013-comp}. In this way, we recover the well-posedness result of \cite%
{Chu2013-comp} (in the linear and nonlinear cases, with $\mathcal{O}$ bounded) by
simply letting $\mathbf{U}\equiv \mathbf{0}$. (Note that in Section \ref%
{static} we discuss the relation between the weak and strong forms of the 
\emph{stationary problem} associated with \eqref{system1}--\eqref{IM2}, and
in Section \ref{nonlinear} we discuss the presence of plate nonlinearity, in
line with \cite{Chu2013-comp}.)
\end{remark}

Finally, we describe the energy balance equation for semigroup solutions to %
\eqref{system1}--\eqref{IC_2}. We introduce the natural notion of \emph{energy} into
the analysis. Semigroup solutions obtained on the finite energy space $%
\mathcal{H}$ are measured in the finite energy norm, which provides us with
the energy functional: for $\mathbf{y}_{0}=(p_{0},\mathbf{u}%
_{0},w_{0},w_{1})\in \mathcal{H}$, we have 
\begin{equation*}
\mathcal{E}(\mathbf{y}_{0})=\frac{1}{2}||\mathbf{y}_{0}||_{\mathcal{H}}^{2}=%
\frac{1}{2}\Big\{||p_{0}||_{\Omega }^{2}+||\mathbf{u}_{0}||_{\Omega
}^{2}+||\Delta w_{0}||_{\Omega }^{2}+||w_{1}||_{\Omega }^{2}\Big\}.
\end{equation*}%
Let us also introduce the convenient notation: 
\begin{equation}
a_{\mathcal{O}}(\mathbf{u},{\boldsymbol{\psi }})=(\sigma (\mathbf{u}%
),\epsilon ({\boldsymbol{\psi }}))_{\mathcal{O}}+\eta (\mathbf{u},%
\boldsymbol{\psi })_{\mathcal{O}}.  \label{bi}
\end{equation}

With strong solutions in hand (corresponding to smooth data in $D(\mathcal A)$), we may test %
\eqref{system1}--\eqref{IC_2} w\i th $p,\mathbf{u},$ and $w_{t}$
(respectively) to obtain the energy balance. The energy balance is then
obtained for semigroup (mild) solutions through the standard limiting
process. Equivalently, it is admissible to test with semigroup solutions
(for $\mathbf{y}_{0}\in \mathcal{H}$, $p\in L^{2}\big(0,t;L^{2}(\mathcal{O})%
\big)$, $\mathbf{u}\in L^{2}\big(0,t;\mathbf{L}^{2}(\mathcal{O})\big)$, and $%
w_{t}\in L^{2}\big(0,t;L^{2}(\Omega )\big)$) in the weak form of the problem
\i n \cite{Chu2013-comp}. This also yields the energy balance below.

\begin{lemma}
\label{energybalance} Consider $\mathbf{y}_0=(p_0,\mathbf{u}_0,w_0,w_1) \in 
\mathcal{H}$ and $\mathbf{U }\in \mathbf{V}_0$. Any mild solution $y(t)=e^{%
\mathcal{A }t}\mathbf{y}_0=(p(t),\mathbf{u}(t),w(t),w_t(t))$ to %
\eqref{system1}--\eqref{IC_2} satisfies for $t>0$: 
\begin{align}  \label{balancelaw}
\mathcal{E}\big(p(t),\mathbf{u}(t),w(t),w_t(t)\big) +\int_0^t a_{\mathcal{O}%
}(\mathbf{u}(\tau),\mathbf{u}(\tau)) d\tau =&~ \mathcal{E}\big(p_0,\mathbf{u}%
_0,w_0,w_1\big) \\
&+\frac{1}{2}\int_0^t\int_{\mathcal{O}} \text{div}(\mathbf{U})[|p(\tau)|^2+|%
\mathbf{u}(\tau)|^2]dx d\tau.  \notag
\end{align}
\end{lemma}

\begin{remark}
We note two features of the energy identity: first, when the field $\mathbf{%
U }\in \mathbf{V}_0$ is also divergence free, the energy identity remains
the same as in the case where $\mathbf{U }\equiv 0$ (like \cite{Chu2013-comp}%
). Secondly, the dissipation integral $\int_0^{\infty} a_{\mathcal{O}}(%
\mathbf{u},\mathbf{u})d\tau$ depends on the quantity $\text{\emph{div}}~%
\mathbf{U}$ as well: again, with $\text{\emph{div}}~\mathbf{U }\equiv 0$, we
see that 
\begin{equation*}
\int_0^{\infty} a_{\mathcal{O}}(\mathbf{u}(\tau),\mathbf{u}(\tau))d\tau <
+\infty,
\end{equation*}
with a bound that depends only on the initial data.
\end{remark}

We conclude this section by noting that we  provide a
discussion of solutions in the presence of nonlinear (von Karman) plate
dynamics, including well-posedness, energy-balance, and stationary
solutions, but we relegate this discussion to Section
\ref{nonlinear}.

\section{Discussion of Main Results in Relation to the Literature}

\label{techreview}

The model under consideration describes the case of a (possibly viscous) 
\emph{compressible} gas/fluid flow and was recently studied in \cite%
{Chu2013-comp} in the case with \emph{zero} speed ($\mathbf{U}=0$) of the
unperturbed flow. Beginning with compressible Navier-Stokes, one can obtain
several fluid-plate cases which are important from an applied point of view:

\begin{itemize}
\item \textbf{Incompressible Fluid}, i.e., $\text{div}~\mathbf{u}=0$ and
density constant: In the \emph{viscous} case, the standard linearized
Navier-Stokes equations arise; fluid-plate interactions in this case were
studied in \cite{ChuRyz2011,cr-full-karman,ChuRyz2012-pois,berlin11}.
Results on well-posedness and attractors for different elastic descriptions
and domains were obtained. In this case, we also mention the work \cite%
{clark,george1,george2} which addresses semigroup well-posedness of a
related linear fluid-plate model, and decay rates via \emph{frequency domain}
techniques. The \emph{inviscid} case was studied in \cite{Chu2013-inviscid}
in the same context.

\item \textbf{Compressible Fluid}: In the inviscid case we can obtain
wave-type dynamics for the (perturbed) velocity potential $(\mathbf{u}%
=\nabla \phi $, potential flow) of the form (see also \cite%
{BA62,bolotin,dowell1}): 
\begin{equation}
\begin{cases}
(\partial _{t}+\mathbf{U}\cdot \nabla )^{2}\phi =\Delta \phi & \text{ in }%
\mathcal{O}\times (0,T), \\ 
{\partial _{z}}\phi =L(w_{t},\nabla w) & \text{ on }\Omega \times (0,T) \\ 
{\partial _{z}}\phi =0 & \text{ on }\partial \mathcal{O}\setminus \Omega
\times (0,T).%
\end{cases}
\label{flow}
\end{equation}%
In these variables, the pressure/density of the fluid has the form $%
p=(\partial _{t}+\mathbf{U}\cdot \nabla )\phi $. Due to the impermeability
assumption, in the case of the perfect fluid, we have only one Neumann-type
boundary condition given above via the operator $L$. The (semigroup)
well-posedness \cite{supersonic, webster} and stability properties \cite%
{springer,delay,conequil2} of this model have been intensively studied.

The \emph{viscous} case was studied in \cite{Chu2013-comp}, and is the
motivation of the current work.
\end{itemize}

In all the papers cited above, the interactive dynamics between fluid and a
plate (or shell) are considered. These analyses are distinguished from those
for other fluid-structure interactive PDE models in that the elastic
structure is two dimensional, and evolves on the boundary of the three
dimensional fluid domain. One of the key issues for the present
configuration---and indeed, one of the main points in the bulk of the
literature above---is the determination of how, and to what extent, the
fluid (de)stabilizes the structure. In \cite%
{Chu2013-comp,ChuRyz2011,ChuRyz2012-pois}, after obtaining well-posedness of
the models (with structural nonlinearity), the existence of compact global
attractors for the dynamics is shown; in some cases the existence of this
invariant set is due strictly to the presence of the fluid, rather than some
underlying structural phenomenon.\footnote{%
We mention that one of the prominent tools utilized in these fluid-structure
interactions---with nonlinearity present in the structure---is the recently
developed \emph{quasi-stability} theory for dissipative dynamical systems
(see \cite{newigor,springer}).} In addition, it is sometimes possible
(perhaps under additional assumptions) to show strong stabilization to
equilibrium for the fluid-structure dynamics (e.g., \cite{conequil2,springer}%
). In all cases where the ambient flow field $\mathbf{U}\neq \mathbf{0}$,
the stability properties of the model depend greatly on the structure and
magnitude of the flow field $\mathbf{U}$ \cite{spectral}. This will
certainly be the case for the dynamics considered here, as one can see from %
\eqref{balancelaw}.

\begin{remark}
The survey papers \cite{berlin11,dcds} provides a nice overview of the
modeling, well-posedness, and long-time behavior results for the family of
dynamics described above.
\end{remark}

We emphasize that in any study involving \emph{compressible fluids}, the
enforced compressibility produces additional density/pressure variables,
and, as a result, well-posedness cannot be obtained in a straightforward
way. In fact, the primary difficulty lies in showing the maximality (range)
condition of the generator, since one has to address this density/pressure
component. This variable cannot be readily eliminated, and therefore
accounts for an elliptic equation which must be solved. To overcome this, we
develop a methodology based on the application of a static well-posedness
result given in the Appendix of \cite{dV} (see also \cite{LaxPhil}). That
paper, as well as \cite{valli}, deals with the stationary compressible
Navier-Stokes equations. Their principal result (obtained independently,
through different methodologies) is a small data well-posedness for the
fully nonlinear fluid problem. However, both approaches first necessarily
provide a framework for the linearized problem; in particular, \cite{dV}
provides a strategy for our analysis of the stationary compressible
fluid-structure PDE which is associated with maximality of the generator.

The pioneering work \cite{Chu2013-comp}, which we cite as the primary
motivating reference, considers the model presented here with $\mathbf{U}%
\equiv 0$. In this paper, solvability and dynamical properties of the model
are considered in the case of a general (possibly unbounded) smooth domain
and in the presence of plate nonlinearity. Along with the well-posedness
result, the existence of a finite dimensional compact global attractor is
proved when the domain is bounded. The techniques used are consistent with
those in \cite{ChuRyz2011,cr-full-karman,Chu2013-inviscid}, namely,
Galerkin-type procedures are implemented, along with good a priori
estimates, in order to produce solutions. As with many fluid-structure
interactions, the critical issue in \cite{Chu2013-comp} is the appearance of
ill-defined traces at the interface. In the incompressible case, one can
recover negative Sobolev trace regularity of the pressure $p$ at the
interface via properties of the Stokes' operator. However, in the viscous
compressible case this is no longer true. Our semigroup approach does not
require the use of approximate solutions. Indeed, we overcome the key
difficulty of trace regularity issues by exploiting cancellations at the
level of solutions with data in the generator. In this way we do not have to
work component-wise on the dynamic equations, though we must work carefully
(and component-wise) on the static problem associated with maximality of the
generator. We remark that, despite these trace regularity issues, when $%
\mathbf{U}\equiv 0$, uniform decay of finite energy solutions is obtained in 
\cite{Chu2013-comp} through a clever Lyapunov approach that makes use of a
Neumann lifting map with associated estimates; this construction is
fundamentally obstructed by the addition of the $\mathbf{U}\cdot \nabla p$
term in the pressure equation here. See the forthcoming work \cite{preprint} on
the decay properties of the model considered here.

\section{The Proof of Theorem \protect\ref{wellp}}

\label{proof}

Our proof of well-posedness hinges on showing that the matrix $\mathcal{A}:D(%
\mathcal{A})\subset \mathcal{H}\rightarrow \mathcal{H}$ generates a $C_{0}$%
-semigroup. At this point, we should note that due to the existence of the
generally nonzero ambient vector field $\mathbf{U}$ in the model, we have a
lack of dissipativity of the operator $\mathcal{A}$. Accordingly, we
introduce the following bounded perturbation $\widehat{\mathcal{A}}$ of our
generator $\mathcal{A}$: 
\begin{equation}
\widehat{\mathcal{A}}=\mathcal{A}-\dfrac{\text{div}(\mathbf{U})}{2}%
\begin{bmatrix}
I & 0 & 0 & 0 \\ 
0 & I & 0 & 0 \\ 
0 & 0 & 0 & 0 \\ 
0 & 0 & 0 & 0%
\end{bmatrix}%
\text{, \ \ }D(\widehat{\mathcal{A}})=D(\mathcal{A}).  \label{pertA}
\end{equation}%
Therewith, the proof of Theorem \ref{wellp} is geared towards establishing
the maximal dissipativity of the linear operator $\widehat{\mathcal{A}}$;
subsequently, an application of the Lumer-Phillips Theorem will yield that $%
\widehat{\mathcal{A}}$ generates a $C_{0}$ semigroup of contractions on $%
\mathcal{H}$. In turn, applying the standard perturbation result \cite{kato}
(given, for instance, in \cite[Theorem 1.1, p.76]{pazy}) yields semigroup
generation for the original modeling fluid-structure operator $\mathcal{A}$
of (\ref{AAA}), via (\ref{pertA}).

\subsection{Dissipativity}

\label{diss} Considering the inner-product for the state space $\mathcal{H}$
given in (\ref{innerp}), for any $\mathbf{y}=[p_{0},\mathbf{u}%
_{0},w_{1},w_{2}]^{T}\in D(\mathcal{A})$ we have 
\begin{align}
\left( \widehat{\mathcal{A}}%
\begin{bmatrix}
p_{0} \\ 
\mathbf{u}_{0} \\ 
w_{1} \\ 
w_{2}%
\end{bmatrix}%
,%
\begin{bmatrix}
p_{0} \\ 
\mathbf{u}_{0} \\ 
w_{1} \\ 
w_{2}%
\end{bmatrix}%
\right) _{\mathcal{H}}=& ~-(\mathbf{U}\cdot \nabla p_{0},p_{0})_{\mathcal{O}%
}-\frac{1}{2}(\text{div}(\mathbf{U})p_{0},p_{0})_{\mathcal{O}}-(\text{div}(%
\mathbf{u}_{0}),p_{0})_{\mathcal{O}}-\eta ||\mathbf{u}_{0}||_{\mathbf{L}^{2}(%
\mathcal{O})}^{2}  \notag \\
& +(\text{div}~\sigma (\mathbf{u}_{0})-\nabla p_{0},\mathbf{u}_{0})_{%
\mathcal{O}}-(\mathbf{U}\cdot \nabla \mathbf{u}_{0},\mathbf{u}_{0})_{%
\mathcal{O}}-\frac{1}{2}(\text{div}(\mathbf{U})\mathbf{u}_{0},\mathbf{u}%
_{0})_{\mathcal{O}}  \notag \\
&  \notag \\
& +(\Delta w_{2},\Delta w_{1})_{\Omega }-(\Delta ^{2}w_{1},w_{2})_{\Omega
}-( \left[ 2\nu \partial _{x_{3}}(\mathbf{u}_{0})_{3}+\lambda \text{ div}(%
\mathbf{u}_{0})\right] _{\Omega }-\left. p_{0}\right\vert _{\Omega
},w_{2})_{\Omega }.  \label{diss_1}
\end{align}

\smallskip

\noindent Applying Green's Theorems to right hand side, we subsequently have 
\begin{align}
\left( \widehat{\mathcal{A}}%
\begin{bmatrix}
p_{0} \\ 
\mathbf{u}_{0} \\ 
w_{1} \\ 
w_{2}%
\end{bmatrix}%
,%
\begin{bmatrix}
p_{0} \\ 
\mathbf{u}_{0} \\ 
w_{1} \\ 
w_{2}%
\end{bmatrix}%
\right) _{\mathcal{H}}=& ~-(\mathbf{U}\cdot \nabla p_{0},p_{0})_{\mathcal{O}%
}-\frac{1}{2}(\text{div}(\mathbf{U})p_{0},p_{0})_{\mathcal{O}}-(\text{div}(%
\mathbf{u}_{0}),p_{0})_{\mathcal{O}}  \notag \\
& -(\sigma (\mathbf{u}_{0}),\epsilon (\mathbf{u}_{0}))_{\mathcal{O}}-\eta ||%
\mathbf{u}_{0}||_{\mathbf{L}^{2}(\mathcal{O})}^{2}+(p_{0},\text{div}(\mathbf{%
u}_{0}))_{\mathcal{O}}+\left\langle \sigma (\mathbf{u}_{0})\mathbf{n}-p_{0}%
\mathbf{n},\mathbf{u}_{0}\right\rangle _{\partial \mathcal{O}}  \notag \\
&  \notag \\
& -(\mathbf{U}\cdot \nabla \mathbf{u}_{0},\mathbf{u}_{0})_{\mathcal{O}}-%
\frac{1}{2}(\text{div}(\mathbf{U})\mathbf{u}_{0},\mathbf{u}_{0})_{\mathcal{O}%
}  \notag \\
&  \notag \\
& +(\Delta w_{2},\Delta w_{1})_{\Omega }-(\Delta w_{1},\Delta w_{2})_{\Omega
}-(\left[ 2\nu \partial _{x_{3}}(\mathbf{u}_{0})_{3}+\lambda \text{div}(%
\mathbf{u}_{0})\right] _{\Omega }-\left. p_{0}\right\vert _{\Omega
},w_{2})_{\Omega }.  \label{dissi_1.2}
\end{align}%
Invoking now the boundary conditions (A.iv) and (A.v), in the definition of
the domain $D(\mathcal{A})$, there is then a cancellation of boundary terms
so as to have%
\begin{align}
\left( \widehat{\mathcal{A}}%
\begin{bmatrix}
p_{0} \\ 
\mathbf{u}_{0} \\ 
w_{1} \\ 
w_{2}%
\end{bmatrix}%
,%
\begin{bmatrix}
p_{0} \\ 
\mathbf{u}_{0} \\ 
w_{1} \\ 
w_{2}%
\end{bmatrix}%
\right) _{\mathcal{H}}=& ~-(\mathbf{U}\cdot \nabla p_{0},p_{0})_{\mathcal{O}%
}-\frac{1}{2}(\text{div}(\mathbf{U})p_{0},p_{0})_{\mathcal{O}}-2i\func{Im}(%
\text{div}(\mathbf{u}_{0}),p_{0})_{\mathcal{O}}  \notag \\
& -(\sigma (\mathbf{u}_{0}),\epsilon (\mathbf{u}_{0}))_{\mathcal{O}}-\eta ||%
\mathbf{u}_{0}||_{\mathbf{L}^{2}(\mathcal{O})}^{2}-(\mathbf{U}\cdot \nabla 
\mathbf{u}_{0},\mathbf{u}_{0})_{\mathcal{O}}-\frac{1}{2}(\text{div}(\mathbf{U%
})\mathbf{u}_{0},\mathbf{u}_{0})_{\mathcal{O}}  \notag \\
&  \notag \\
& -2i\func{Im}(\Delta w_{1},\Delta w_{2})_{\Omega }.  \label{diss_1.4}
\end{align}

Moreover, via Green's Theorem, as well as the assumption that $\mathbf{U}\in 
\mathbf{V}_{0}$ (as defined in (\ref{V_0})), we obtain 
\begin{equation}
2\func{Re}(\mathbf{U}\cdot \nabla p_{0},p_{0})_{\mathcal{O}}=-\int_{\mathcal{%
O}}\text{div}(\mathbf{U})\left\vert p_{0}\right\vert ^{2}d\mathcal{O};
\label{diss_2}
\end{equation}%
\begin{equation}
2\func{Re}(\mathbf{U}\cdot \nabla \mathbf{u}_{0},\mathbf{u}_{0})_{\mathcal{O}%
}=-\int_{\mathcal{O}}\text{div}(\mathbf{U})\left\vert \mathbf{u}%
_{0}\right\vert ^{2}d\mathcal{O}.  \label{diss_3}
\end{equation}

\bigskip

\noindent Applying these relations to the right hand of (\ref{diss_1.4}), we
then have%
\begin{equation*}
\func{Re}\left( \widehat{\mathcal{A}}%
\begin{bmatrix}
p_{0} \\ 
\mathbf{u}_{0} \\ 
w_{1} \\ 
w_{2}%
\end{bmatrix}%
,%
\begin{bmatrix}
p_{0} \\ 
\mathbf{u}_{0} \\ 
w_{1} \\ 
w_{2}%
\end{bmatrix}%
\right) _{\mathcal{H}}=-(\sigma (\mathbf{u}_{0}),\epsilon (\mathbf{u}_{0}))_{%
\mathcal{O}}-\eta ||\mathbf{u}_{0}||_{\mathbf{L}^{2}(\mathcal{O})}^{2}\leq 0,
\end{equation*}%
which establishes the dissipativity of $\widehat{\mathcal{A}}:D(\mathcal{A}%
)\subset \mathcal{H}\rightarrow \mathcal{H}.$

\subsection{Maximality}

\label{max} In this section we show the maximality property of the operator $%
\widehat{\mathcal{A}}$ on the space $\mathcal{H}$. To this end, we will need
to establish the \emph{range condition}, at least for parameter $\xi >0~~$%
sufficiently large. Namely, we must show 
\begin{equation}
Range(\xi I-\widehat{\mathcal{A}})=\mathcal{H},~~\text{for some}~~\xi >0.
\label{range_0}
\end{equation}%
This necessity is equivalent to finding $[p,{\mathbf{v}},w_{1},w_{2}]\in D(%
\mathcal{A})$ which satisfies, for given $[p^{\ast },{\mathbf{v}}^{\ast
},w_{1}^{\ast },w_{2}^{\ast }]\in \mathcal{H}$, the abstract equation 
\begin{equation}
(\xi I-\widehat{\mathcal{A}})%
\begin{bmatrix}
p \\ 
{\mathbf{v}} \\ 
w_{1} \\ 
w_{2}%
\end{bmatrix}%
=%
\begin{bmatrix}
p^{\ast } \\ 
{\mathbf{v}}^{\ast } \\ 
w_{1}^{\ast } \\ 
w_{2}^{\ast }%
\end{bmatrix}%
.  \label{range}
\end{equation}%
Given the definition of $\mathcal{A}$ in (\ref{AAA}), then in PDE terms,
solving the abstract equation (\ref{range_0}) is equivalent to proving that
the following system of equations, with given data $[p^{\ast },{\mathbf{v}}%
^{\ast },w_{1}^{\ast },w_{2}^{\ast }]\in \mathcal{H}$, has a (unique)
solution $[p,{\mathbf{v}},w_{1},w_{2}]\in D(\mathcal{A})$:%
\begin{align}
& \left\{ 
\begin{array}{l}
\xi p+\mathbf{U}\cdot \nabla p+\frac{1}{2}\text{div}(\mathbf{U})p+\text{div}(%
{\mathbf{v}})=\text{ }p^{\ast }~~\text{ in }~\mathcal{O} \\ 
\xi {\mathbf{v}}+\mathbf{U}\cdot \nabla {\mathbf{v}}+\frac{1}{2}\text{div}(%
\mathbf{U}){\mathbf{v}}-\text{div}~\sigma ({\mathbf{v}})+\eta {\mathbf{v}}%
+\nabla p=~{\mathbf{v}}^{\ast }~~\text{ in }~\mathcal{O} \\ 
(\sigma (\mathbf{v})\mathbf{n}-p\mathbf{n})\cdot \boldsymbol{\tau }=0~\text{
on }~\partial \mathcal{O} \\ 
\mathbf{v}\cdot \mathbf{n}=0~\text{ on }~S \\ 
\mathbf{v}\cdot \mathbf{n}=w_{2}~\text{ on }~\Omega%
\end{array}%
\right.  \label{statice1} \\
&  \notag \\
& \left\{ 
\begin{array}{l}
\xi w_{1}-w_{2}=~w_{1}^{\ast }~~\text{ on }~\Omega \\ 
\xi w_{2}+\Delta ^{2}w_{1}+\left[ 2\nu \partial _{x_{3}}(\mathbf{v)}%
_{3}+\lambda \text{div}(\mathbf{v})-p\right] _{\Omega }=w_{2}^{\ast }~\text{
on }~\Omega \\ 
w_{1}=\frac{\partial w_{1}}{\partial \nu }=0~\text{ on }~\partial \Omega .%
\end{array}%
\right.  \label{staticsys3.5}
\end{align}%
We will give our proof of maximality in two steps. In the first step, we
will show the existence and uniqueness of \textquotedblleft
uncoupled\textquotedblright\ versions of the compressible fluid-structure
PDE system (\ref{statice1}) which is satisfied by variables $\left\{ p,%
\mathbf{v}\right\} $. To this end, the key ingredient will be the
well-posedness result Theorem \ref{dV}, which is applicable to (uncoupled)
equations of the type satisfied by the pressure variable. (See also \cite{dV}
and \cite{LaxPhil}.)

Subsequently, we proceed to establish the range condition (\ref{range}), by
sequentially proving the existence of the pressure-fluid-structure
components $\left\{ p,\mathbf{v},w_{1},w_{2}\right\} $ which solve the
coupled system (\ref{statice1})--(\ref{staticsys3.5}). This work for
pressure-fluid-structure static well-posedness involves appropriate uses of
the Lax-Milgram Theorem.

\medskip

\noindent \textbf{STEP 1}:\smallskip\newline
Consider the following $\xi $-parameterized PDE system on the fluid domain $%
\mathcal{O}$, with given forcing terms $\left\{ p^{\ast },{\mathbf{v}}^{\ast
}\right\} \in {L}^{2}(\mathcal{O})\times \lbrack \mathbf{V}_{0}]^{\prime }$%
\textbf{\ }and boundary data $g\in H_{0}^{1/2+\epsilon }(\Omega )$, where $%
\epsilon >0$.

\begin{align}
\xi p+\mathbf{U}\cdot \nabla p+\frac{1}{2}\text{div}(\mathbf{U})p+\text{div}(%
{\mathbf{v}})=& ~p^{\ast }~~\text{ in }~\mathcal{O}  \label{a1} \\
\xi {\mathbf{v}}+\mathbf{U}\cdot \nabla {\mathbf{v}}+\frac{1}{2}\text{div}(%
\mathbf{U}){\mathbf{v}}-\text{div}~\sigma ({\mathbf{v}})+\eta {\mathbf{v}}%
+\nabla p=& ~{\mathbf{v}}^{\ast }~~\text{ in }~\mathcal{O}  \label{a2} \\
\left( \sigma ({\mathbf{v}})\mathbf{n}-p\mathbf{n}\right) \cdot \boldsymbol{%
\tau }=& ~0~~\text{ on }~\partial \mathcal{O}  \label{a3} \\
{\mathbf{v}}\cdot \mathbf{n}=& ~0~~\text{ on }~S  \label{a4} \\
{\mathbf{v}}\cdot \mathbf{n}=& ~g~~\text{ on}~\Omega  \label{a5}
\end{align}
(and, where again, the ambient vector field $\mathbf{U}\in \mathbf{V}%
_{0}\cap \mathbf{H}^{3}(\mathcal{O})$).

\smallskip

\noindent {STEP 1} consists of proving the following (driving) lemma for the
existence and uniqueness of the solution $\left\{p,\mathbf{v}\right\} $ of (%
\ref{a1})--(\ref{a5}).

\begin{lemma}
\label{staticwellp} (i) With reference to problem (\ref{a1})--(\ref{a5}):
with given data 
\begin{equation*}
\lbrack p^{\ast },{\mathbf{v}}^{\ast },g]\in {L}^{2}(\mathcal{O})\times
\lbrack \mathbf{V}_{0}]^{\prime }\times H_{0}^{\frac{1}{2}+\epsilon }(\Omega
)\text{,}
\end{equation*}
and with $\xi >0$ sufficiently large, there exists a unique solution $%
\left\{p,\mathbf{v}\right\} $ $\in {L}^{2}(\mathcal{O})\times \mathbf{H}^{1}(%
\mathcal{O})$ of (\ref{a1})--(\ref{a5}).

(ii) The fluid solution component ${\mathbf{v}}$ is of the form 
\begin{equation}
\mathbf{v}=\mathbf{u}+\widetilde{{\mathbf{v}}_{0}}\text{,}  \label{crucial}
\end{equation}%
where $\mathbf{u}\in \mathbf{V}_{0}$, and $\widetilde{{\mathbf{v}}_{0}}\in 
\mathbf{H}^{1}(\mathcal{O})$ satisfies%
\begin{equation}
\widetilde{{\mathbf{v}}_{0}}\Big|_{\partial \mathcal{O}}=%
\begin{cases}
0 & ~\text{ on }~S \\ 
g\mathbf{n} & ~\text{ on }~\Omega .%
\end{cases}
\label{crucial2}
\end{equation}%
(iii) The trace term $\left[ \sigma ({\mathbf{v}})\mathbf{n}-p\mathbf{n}%
\right] _{\partial \mathcal{O}}\in \mathbf{H}^{-\frac{1}{2}}(\partial 
\mathcal{O})$, and moreover satisfies 
\begin{equation}
\left\langle \sigma ({\mathbf{v}})\mathbf{n}-p\mathbf{n,\tau }\right\rangle
_{\mathbf{H}^{-\frac{1}{2}}(\partial \mathcal{O})\times \mathbf{H}^{\frac{1}{%
2}}(\partial \mathcal{O})}=0\text{ \ for all }\mathbf{\tau }\in
TH^{1/2}(\partial \mathcal{O}),  \label{crucial3}
\end{equation}%
and so the boundary condition (\ref{a3}) is satisfied in the sense of
distributions; see (\ref{weak2}) of Remark \ref{weaker}

(iv) The fluid and pressure solution components $(p,\mathbf{v})$ satisfies
the following estimates, for $\xi =\xi (\mathbf{U})$ large enough: 
\begin{eqnarray}
\left\Vert p\right\Vert _{L^2(\mathcal{O})} &\leq &\frac{C}{\xi }\left\Vert
[p^{\ast },{\mathbf{v}}^{\ast },g]\right\Vert _{\mathbf{L}^{2}(\mathcal{O}%
)\times \lbrack \mathbf{V}_{0}]^{\prime }\times H_{0}^{\frac{1}{2}+\epsilon
}(\Omega )};  \label{cdd_0} \\
\left\Vert {\mathbf{v}}\right\Vert _{\mathbf{H}^{1}(\mathcal{O})} &\leq
&C\left\Vert [p^{\ast },{\mathbf{v}}^{\ast },g]\right\Vert _{\mathbf{L}^{2}(%
\mathcal{O})\times \lbrack \mathbf{V}_{0}]^{\prime }\times H_{0}^{\frac{1}{2}%
+\epsilon }(\Omega )}.  \label{cdd}
\end{eqnarray}
\end{lemma}

\begin{proof}[Proof of Lemma \ref{staticwellp}] We give the proof in two parts.
Our beginning point is to resolve the pressure term; this will be
accomplished by applying Theorem \ref{dV} of the Appendix. To this end: If
we initially consider the equation 
\begin{equation}
\xi p+\mathbf{U}\cdot \nabla p+\frac{1}{2}\text{div}(\mathbf{U})p=\sigma ~~%
\text{ in }~\mathcal{O},  \label{above}
\end{equation}%
where $\sigma \in L^{2}(\mathcal{O})$ and $\mathbf{U}\in \mathbf{V}_{0}\cap 
\mathbf{H}^{3}(\mathcal{O})$, we have by Theorem \ref{dV} the existence of a
unique $L^{2}(\mathcal{O})$-function $p$ which is a weak solution of (\ref%
{above}); namely, it satisfies the variational relation%
\begin{equation}
\xi \int_{\mathcal{O}}p\phi d\mathbf{x}-\int_{\mathcal{O}}pdiv(\phi \mathbf{U%
})d\mathbf{x+}\frac{1}{2}\int_{\mathcal{O}}div(\mathbf{U})p\phi =\int_{%
\mathcal{O}}\sigma \phi d\mathbf{x}\text{, \ for all }\phi \in H^{1}(%
\mathcal{O})  \label{weak}
\end{equation}%
(and in particular for $\phi \in \mathcal{D}(\mathcal{O})$; so we infer that
for given $L^{2}$-function $\sigma $, corresponding $L^{2}$-solution $p$
satisfies the PDE (\ref{above}) pointwise a.e.)

Moreover, for $\xi =\xi (\mathbf{U})$ sufficiently large we have the
estimate---see (\ref{est}) of Theorem \ref{dV}, 
\begin{equation}
||p||_{L^{2}(\mathcal{O})}\leq \dfrac{1}{\xi }||\sigma ||_{L^{2}(\mathcal{O}%
)}.  \label{large}
\end{equation}

\medskip

With the well-posedness above, in order to find the existence and uniqueness
of the fluid component $\mathbf{v}$, we now turn our attention to (\ref{a1}%
)--(\ref{a5}). In view of the well-posedness of \ (\ref{above}), we
decompose the fluid term $\mathbf{v}$ and pressure term $p$ as follows: 
\begin{equation}
\mathbf{v}=\mathbf{u}+\widetilde{{\mathbf{v}}_{0}};  \label{v_com}
\end{equation}%
\begin{equation}
p=p[\mathbf{u}]+p[\widetilde{{\mathbf{v}}_{0}}]+p[p^{\ast }],  \label{p_com}
\end{equation}%
where $\mathbf{u}\in \mathbf{V}_{0}$ is the new fluid solution variable, and $%
\widetilde{{\mathbf{v}}_{0}}\in \mathbf{H}^{1}(\mathcal{O})$ is a vector
field which is chosen to satisfy 
\begin{equation}
\widetilde{{\mathbf{v}}_{0}}\Big|_{\partial \mathcal{O}}=%
\begin{cases}
0 & ~\text{ on }~S \\ 
g\mathbf{n} & ~\text{ on }~\Omega .%
\end{cases}
\label{bc}
\end{equation}%
To wit: since boundary data $g\in H^{\frac{1}{2} +\epsilon}_0(\Omega)$ -- with necessarily $\epsilon>0$ -- we can extend by zero the function $g\mathbf n\Big|_{\Omega}=g[1,0,0],$ so as to have a $ \mathbf H^{\frac{1}{2} +\epsilon} $-- function on all of $\partial \mathcal{O}.$ (See e.g., Theorem 3.33, p.95 of \cite{Mc}.) In turn, since the Sobolev Dirichlet trace map from $H^{s}(\mathcal{O})$ into $H^{s-\frac{1}{2}}(\partial \mathcal{O})$ is surjective for $\frac{1}{2}<s<\frac{3}{2},$ then the existence of given $\widetilde{{\mathbf{v}}_{0}}\in \mathbf H^{1+\epsilon}(\mathcal{O})$ is assured. (See e.g., Theorem 3.38 of \cite{Mc}, valid for Lipschitz domains.)

\smallskip

Moreover, the functions $p[\mathbf{u}],~p[\widetilde{{\mathbf{v}}_{0}}],$
and $p[p^{\ast }]$ solve the respective versions of (\ref{above}): 
\begin{align}
\xi p[\mathbf{u}]+\mathbf{U}\cdot \nabla p[\mathbf{u}]+\frac{\text{div}(%
\mathbf{U})}{2}p[\mathbf{u}]=& ~-\text{div}(\mathbf{u})  \label{duh1} \\
\xi p[\widetilde{{\mathbf{v}}_{0}}]+\mathbf{U}\cdot \nabla p[\widetilde{{%
\mathbf{v}}_{0}}]+\frac{\text{div}(\mathbf{U})}{2}p[\widetilde{{\mathbf{v}}%
_{0}}]=& ~-\text{div}(\widetilde{{\mathbf{v}}_{0}}) \\
\xi p[p^{\ast }]+\mathbf{U}\cdot \nabla p[p^{\ast }]+\frac{\text{div}(%
\mathbf{U})}{2}p[p^{\ast }]=& ~p^{\ast },
\end{align}
with estimates---for $\xi =\xi (\mathbf{U})$ large enough, see (\ref{large})
and (\ref{bc}) --- 
\begin{align}
||p[\mathbf{u}]||_{L^2(\mathcal{O})}\leq & ~\dfrac{C}{\xi }||\mathbf{u}||_{%
\mathbf{H}^{1}(\mathcal{O})}  \label{first} \\
||p[\widetilde{{\mathbf{v}}_{0}}]||_{L^2(\mathcal{O})}\leq & ~\dfrac{C}{\xi }||%
\widetilde{{\mathbf{v}}_{0}}||_{\mathbf{H}^{1}(\mathcal{O})}  \notag \\
\leq & ~\dfrac{C}{\xi }||g||_{H_{0}^{1/2+\epsilon }(\Omega )}  \label{second}
\\
||p[p^{\ast }]||_{L^2(\mathcal{O})}\leq & ~\dfrac{C}{\xi }||p^{\ast }||_{\mathcal{%
O}},  \label{third}
\end{align}%
where again, fluid variable $\mathbf{u}$ will be sought after.

\bigskip

Now, the rest part of the proof relies on the application of Lax-Milgram
theorem, by way of solving for $\mathbf{u}$ in (\ref{v_com}). For this
reason, we firstly define the operator $A\in \mathcal{L}(\mathbf{V}_{0},[%
\mathbf{V}_{0}]^{\prime })$ to be, for all $\mathbf{\psi }\in \mathbf{V}_{0}$, 
\begin{align}
\left\langle A\mathbf{u},\mathbf{\psi }\right\rangle _{\mathbf{v}_{0}\times
\lbrack \mathbf{v}_{0}]^{\prime }}=& ~\big(\xi \mathbf{u}+\mathbf{U}\cdot
\nabla \mathbf{u}+\frac{1}{2}\text{div}(\mathbf{U})\mathbf{u},\mathbf{\psi }%
\big)_{\mathcal{O}}  \notag \\
& +\big(\sigma (\mathbf{u}),\epsilon (\mathbf{\psi })\big)_{\mathcal{O}%
}+\eta \big(\mathbf{u},\mathbf{\psi }\big)_{\mathcal{O}}-\big(p[\mathbf{u}],%
\text{div}(\mathbf{\psi })\big)_{\mathcal{O}},  \label{A}
\end{align}%
where again, $p[\mathbf{u}]$ solves (\ref{duh1}).

So, with a view of finding the solution pair $(p,\mathbf{v})$, via the
expressions (\ref{v_com}) and (\ref{p_com}), we are led to consider the
following variational problem: Find $\mathbf{u}\in \mathbf{V}_{0}$ which
solves, for every $\mathbf{\psi }\in \mathbf{V}_{0}$,

\begin{equation}
\left\langle A{\mathbf{u}},\mathbf{\psi }\right\rangle _{\mathbf{v}%
_{0}\times \lbrack \mathbf{v}_{0}]^{\prime }}=\left\langle F,\mathbf{\psi }%
\right\rangle _{\mathbf{v}_{0}\times \lbrack \mathbf{v}_{0}]^{\prime }};
\label{vareq1}
\end{equation}%
where the forcing term $F\in \lbrack \mathbf{V}_{0}]^{\prime }$ is given by 
\begin{align}
\left\langle F,\mathbf{\psi }\right\rangle _{\mathbf{v}_{0}\times \lbrack 
\mathbf{v}_{0}]^{\prime }}& =\left( {\mathbf{v}}^{\ast },\mathbf{\psi }%
\right) _{\mathcal{O}}-\big(\sigma (\widetilde{{\mathbf{v}}_{0}}),\epsilon (%
\mathbf{\psi })\big)_{\mathcal{O}}-\eta (\widetilde{{\mathbf{v}}_{0}},%
\mathbf{\psi })_{\mathcal{O}}  \notag \\
& -\big(\xi \widetilde{{\mathbf{v}}_{0}}+\mathbf{U}\cdot \nabla \widetilde{{%
\mathbf{v}}_{0}}+\frac{1}{2}\text{div}(\mathbf{U})\widetilde{{\mathbf{v}}_{0}%
},\mathbf{\psi }\big)_{\mathcal{O}}  \notag \\
& +\big(p[\widetilde{{\mathbf{v}}_{0}}]+p[p^{\ast }],\text{div}(\mathbf{\psi 
})\big)_{\mathcal{O}},  \label{force}
\end{align}%
for given $\mathbf{\psi }\in \mathbf{V}_{0}.$ After considering the
definition of $\mathbf{V}_{0}$ and using the divergence theorem we note that 
\begin{equation}
\big(\mathbf{U}\cdot \nabla \mathbf{\psi }+\frac{1}{2}\text{div}(\mathbf{U})%
\mathbf{\psi },\mathbf{\psi }\big)_{\mathcal{O}}=0,~~\forall ~\mathbf{\psi }%
\in \mathbf{V}_{0}.  \label{there}
\end{equation}%
Moreover, by estimate (\ref{first}) we have for all $\mathbf{\psi }\in 
\mathbf{V}_{0}$, 
\begin{equation}
\Big\vert\left( p[\mathbf{\psi }],\func{div}(\mathbf{\psi })\right) _{%
\mathcal{O}}\Big\vert\leq \frac{C}{\xi }||\mathbf{\psi }||_{\mathbf{H}^{1}(%
\mathcal{O})}^{2}.  \label{with}
\end{equation}

Combining (\ref{there}) and (\ref{with}) with Korn's inequality---see e.g.,
Theorem 2.6.5, p.93 of \cite{kesavan}---we then have, for $\xi >0$
sufficiently large, 
\begin{equation}
\begin{array}{l}
\left\langle A\mathbf{\psi },\mathbf{\psi }\right\rangle _{\mathbf{v}%
_{0}\times \lbrack \mathbf{v}_{0}]^{\prime }}= 
(\xi +\eta )||\mathbf{\psi }||_{\mathcal{O}}^{2}+\sigma (\mathbf{\psi }%
),\epsilon (\mathbf{\psi })\big)-\big(p[\mathbf{\psi }],\text{div}(\mathbf{%
\psi })\big)_{\mathcal{O}}\geq ~c||\mathbf{\psi }||_{\mathbf{H}^{1}(\mathcal{%
O})}^{2},~~\forall ~\mathbf{\psi }\in \mathbf{V}_{0},%
\end{array}
\label{V_e}
\end{equation}%
where coercivity constant $c>0$ is independent of sufficiently large $\xi >0$%
. Therefore, $A\in \mathcal{L}(\mathbf{V}_{0},[\mathbf{V}_{0}]^{\prime })$
is $\mathbf{V}_{0}$-elliptic for $\xi >0$ large enough. Consequently, by the
Lax-Milgram Theorem, the variational equation (\ref{vareq1}) has a unique
solution ${\mathbf{u}}\in \mathbf{V}_{0}$, which will in turn yield the
solution pair $(\mathbf{v},p)$ of (\ref{a1})--(\ref{a5} through the
relations (\ref{v_com}) and (\ref{p_com}). In particular, from (\ref{v_com})
and (\ref{bc}), $\mathbf{v}\in \mathbf{H}^{1}(\mathcal{O})$ admits of the
decomposition (\ref{crucial}); and since 
\begin{equation*}
\left. \mathbf{v}\right\vert _{\partial \mathcal{O}}\cdot \mathbf{n}=%
\begin{cases}
0 & ~\text{ on }~S \\ 
g & ~\text{ on }~\Omega ,%
\end{cases}%
\end{equation*}%
the obtained solution component $\mathbf{v}$ satisfies the boundary
conditions (\ref{a4})--(\ref{a5}).

\medskip

In addition, from (\ref{vareq1}), (\ref{v_com}) and (\ref{p_com}), $(p,\mathbf{%
v})$ satisfies the variational relation 
\begin{align}
\big(\xi {\mathbf{v}}+\mathbf{U}\cdot \nabla {\mathbf{v}}+\frac{1}{2}\text{%
div}(\mathbf{U}){\mathbf{v}},\mathbf{\psi }\big)_{\mathcal{O}}+(\sigma ({%
\mathbf{v}}),\epsilon (\mathbf{\psi }))_{\mathcal{O}}+\eta ({\mathbf{v}},%
\mathbf{\psi })& -(p,\text{div}(\mathbf{\psi }))_{\mathcal{O}}  \notag \\
=& ({\mathbf{v}}^{\ast },\mathbf{\psi })_{\mathcal{O}},~~\forall ~\mathbf{%
\psi }\in \mathbf{V}_{0}.  \label{L0}
\end{align}

\noindent In particular, if $\mathbf{\psi }\in \lbrack \mathcal{D}(\mathcal{O%
})]^{3}$, we then have 
\begin{align}
-(\text{div}~\sigma ({\mathbf{v}}),\mathbf{\psi })_{\mathcal{O}}+\eta ({\mathbf{v}},%
\mathbf{\psi })_{\mathcal{O}}& +(\nabla p,\mathbf{\psi })_{\mathcal{O}} 
\notag \\
=& -(\xi {\mathbf{v}}+\mathbf{U}\cdot \nabla {\mathbf{v}}+\frac{1}{2}\func{%
div}(\mathbf{U}){\mathbf{v}},\mathbf{\psi })_{\mathcal{O}}+({\mathbf{v}}%
^{\ast },\mathbf{\psi })_{\mathcal{O}}.  \label{L1}
\end{align}

\noindent Since ${\mathbf{v}}\in \mathbf{H}^{1}(\mathcal{O})$, this relation
and the density of $[\mathcal{D}(\mathcal{O})]^{3}$ in $\mathbf{L}^{2}(%
\mathcal{O})$, yield that 
\begin{equation}
-\func{div}\sigma ({\mathbf{v}})+\eta {\mathbf{v}}+\nabla p=-\big(\xi {%
\mathbf{v}}+\mathbf{U}\cdot \nabla {\mathbf{v}}+\frac{1}{2}\func{div}(%
\mathbf{U}){\mathbf{v}}\big)+{\mathbf{v}}^{\ast }  \label{L2}
\end{equation}%
in $L^{2}$-sense. And so $(p,{\mathbf{v}})$ satisfies the coupled system (%
\ref{a2}) and (\ref{a1}) pointwise. (See the remark below (\ref{weak})).

\medskip

Finally, since $[-\text{div}~\sigma ({\mathbf{v}})+\nabla p]\in \mathbf{L}^{2}(%
\mathcal{O})$, and $(p,{\mathbf{v}})\in L^{2}(\mathcal{O})\times \mathbf{H}^{1}(\mathcal{O})$, a classic integration by parts argument will yield the
following trace regularity: 
\begin{equation}
\lbrack \sigma ({\mathbf{v}})\mathbf{n}-p\mathbf{n}]\in \mathbf{H}%
^{-1/2}(\partial \mathcal{O})  \label{trace}
\end{equation}%
(see e.g., Theorem 13.2.3, p.326, of \cite{aubin}). Integrating by parts in
(\ref{L0}), we consequently have for all $\mathbf{\psi }\in \mathbf{V}_{0}.$ 
\begin{align*}
({\mathbf{v}}^{\ast },\mathbf{\psi })_{\mathcal{O}}=& ~(\xi {\mathbf{v}}+%
\mathbf{U}\cdot \nabla {\mathbf{v}}+\frac{1}{2}\text{div}(\mathbf{U}){%
\mathbf{v}},\mathbf{\psi })_{\mathcal{O}}-(\func{div}\sigma ({\mathbf{v}}%
)-\nabla p,\mathbf{\psi })_{\mathcal{O}} \\
& +\eta ({\mathbf{v}},\mathbf{\psi })_{\mathcal{O}}+\langle \sigma ({\mathbf{%
v}})\mathbf{n-}p\mathbf{n},\mathbf{\psi }\rangle _{\partial \mathcal{O}},
\end{align*}%
or, after invoking (\ref{L2}): 
\begin{equation*}
\langle \sigma ({\mathbf{v}})\mathbf{n}-p\mathbf{n},\mathbf{\psi }\rangle
_{\partial \mathcal{O}}=0\text{ for every }\mathbf{\psi }\in \mathbf{V}_{0}.
\end{equation*}%
This orthogonality and the surjectivity of the trace mapping from $\mathbf{H}%
^{1}(\mathcal{O})\rightarrow \mathbf{H}^{1/2}(\partial \mathcal{O})$, allow
us to deduce that the obtained solution pair $(p,{\mathbf{v}})$ satisfies
the boundary condition 
\begin{equation*}
\lbrack \sigma ({\mathbf{v}})\mathbf{n}-p\mathbf{n}]\cdot \boldsymbol{\tau }%
=0\text{ for all }\boldsymbol{\tau }\in TH^{1/2}(\partial \mathcal{O}),
\end{equation*}%
in \emph{weak sense}; i.e., as in (\ref{weak2}) of Remark \ref{weaker}.

\smallskip

Lastly, the necessary continuous dependence estimates (\ref{cdd_0})--(\ref%
{cdd}) come from collecting (\ref{p_com}) and (\ref{first})--(\ref{third})
-- for $p$ in (\ref{cdd_0}); and (\ref{v_com}), (\ref{bc}), (\ref{V_e}), and
(\ref{force}) -- for $\mathbf v$ in (\ref{cdd}). This concludes the proof of Lemma %
\ref{staticwellp}, and so {STEP 1} of the maximality argument. \end{proof}

\bigskip

\noindent \textbf{STEP 2}\smallskip\newline
With Lemma \ref{staticwellp} in hand, we properly deal with the coupled
fluid-structure PDE system (\ref{statice1})--(\ref{staticsys3.5}). Our
solution here will be predicated on finding the structural variable $w_{1}$
which solves the $\Omega $-problem 
\begin{align}
\xi ^{2}w_{1}+\Delta ^{2}w_{1}+[2\nu \partial _{x_{3}}({\mathbf{v}}%
)_{3}]_{\Omega }+\lambda \text{div}({\mathbf{v}})|_{\Omega }-p|_{\Omega }=&
~w_{2}^{\ast }+\xi w_{1}^{\ast }~~\text{ on}~\Omega  \label{staticsys5} \\
w_{1}\big |_{\partial \Omega }=& ~\partial _{\nu }w_{1}\big|_{\partial
\Omega }=0.  \label{staticsys6}
\end{align}

\medskip

Let $(p^{\ast },{\mathbf{v}}^{\ast })\in L^{2}(\mathcal{O})\times \mathbf{L}%
^{2}(\mathcal{O})$ be the pressure and fluid data from (\ref{statice1}). Let 
$z\in H_{0}^{2}(\Omega )$ be given. Then from Lemma \ref{staticwellp}, we
know that the following problem has a unique solution $\left\{ p(z;p^{\ast };%
\mathbf{v}^{\ast }), \mathbf{v}(z;p^{\ast };\mathbf{v}^{\ast })\right\} $: 
\begin{equation}
\begin{array}{l}
\xi p(z;p^{\ast };\mathbf{v}^{\ast })+\mathbf{U}\cdot \nabla p(z;p^{\ast };%
\mathbf{v}^{\ast })+\frac{1}{2}\text{div}(\mathbf{U})p(z;p^{\ast };\mathbf{v}%
^{\ast })+\text{div}({\mathbf{v}}(z;p^{\ast };\mathbf{v}^{\ast }))=~p^{\ast
}~~\text{ in }~\mathcal{O} \\ 
\xi {\mathbf{v}}(z;p^{\ast };\mathbf{v}^{\ast })+\mathbf{U}\cdot \nabla {%
\mathbf{v}}(z;p^{\ast };\mathbf{v}^{\ast })+\frac{1}{2}\text{div}(\mathbf{U})%
{\mathbf{v}}(z;p^{\ast };\mathbf{v}^{\ast }) \\ 
\text{ \ \ \ \ \ \ \ \ \ \ \ \ \ \ \ \ \ \ \ }-\text{div}~\sigma ({\mathbf{v}%
}(z;p^{\ast };\mathbf{v}^{\ast }))+\eta {\mathbf{v}}(z;p^{\ast };\mathbf{v}%
^{\ast })+\nabla p(z;p^{\ast };\mathbf{v}^{\ast })=~{\mathbf{v}}^{\ast }~~%
\text{ in }~\mathcal{O} \\ 
\left( \sigma ({\mathbf{v}}(z;p^{\ast };\mathbf{v}^{\ast }))\mathbf{n}%
-p(z;p^{\ast };\mathbf{v}^{\ast })\mathbf{n}\right) \cdot \boldsymbol{\tau }%
=~0~~\text{ on }~\partial \mathcal{O} \\ 
{\mathbf{v}}(z;p^{\ast };\mathbf{v}^{\ast })\cdot \mathbf{n}=~0~~\text{ on }%
~S \\ 
{\mathbf{v}}(z;p^{\ast };\mathbf{v}^{\ast })\cdot \mathbf{n}=~z~~\text{ on}%
~\Omega .%
\end{array}
\label{staticsys1}
\end{equation}

\medskip

Akin to what was done in {STEP 1}, we decompose the solution of the BVP (\ref%
{staticsys1}) into two parts:%
\begin{eqnarray}
\mathbf{v}(z;p^{\ast };\mathbf{v}^{\ast }) &=&\mathbf{v(z)+v}(p^{\ast };%
\mathbf{v}^{\ast });  \label{d1} \\
p(z;p^{\ast };\mathbf{v}^{\ast }) &=&p(z)+p(p^{\ast };\mathbf{v}^{\ast }),
\label{d2}
\end{eqnarray}

where $(p(z),{\mathbf{v}}(z))\in \mathbf{H}^{1}(\mathcal{O})\times L^{2}(%
\mathcal{O})$ is the solution of the problem 
\begin{align}
\xi p(z)+\mathbf{U}\cdot \nabla p(z)+\frac{1}{2}\text{div}(\mathbf{U})p(z)+%
\text{div}({\mathbf{v}}(z))=& 0~~\text{ in }~\mathcal{O}  \label{huh} \\
\xi {\mathbf{v}(z)}+\mathbf{U}\cdot \nabla {\mathbf{v}}(z)+\frac{1}{2}\text{%
div}(\mathbf{U}){\mathbf{v}}(z)-\text{div}~\sigma ({\mathbf{v}}(z))+\eta {%
\mathbf{v}}(z)+\nabla p(z)=& ~0~~\text{ in }~\mathcal{O} \\
\left( \sigma ({\mathbf{v}}(z))\mathbf{n}-p(z)\mathbf{n}\right) \cdot 
\boldsymbol{\tau }=& ~0~~\text{ on }~\partial \mathcal{O} \\
{\mathbf{v}}(z)\cdot \mathbf{n}=& ~0~~\text{ on }~S \\
{\mathbf{v}}(z)\cdot \mathbf{n}=& ~z~~\text{ on}~\Omega ,  \label{givensys2}
\end{align}%
and $\big(p(p^{\ast },{\mathbf{v}}^{\ast }),{\mathbf{v}}(p^{\ast };{\mathbf{v%
}}^{\ast })\big)\equiv (\overline{p},\overline{{\mathbf{v}}})\in L^{2}(%
\mathcal{O})\times \mathbf{V}_{0}$ ~~is the solution of the problem 
\begin{align}
\xi \overline{p}+\mathbf{U}\cdot \nabla \overline{p}+\frac{1}{2}\text{div}(%
\mathbf{U})\overline{p}+\text{div}\overline{{\mathbf{v}}}=& ~p^{\ast }~\text{
in }~~\mathcal{O}  \label{insert} \\
\xi \overline{{\mathbf{v}}}+\mathbf{U}\cdot \nabla \overline{{\mathbf{v}}}+%
\frac{1}{2}\text{div}(\mathbf{U})\overline{{\mathbf{v}}}-\text{div}~\sigma (%
\overline{{\mathbf{v}}})+\eta \overline{{\mathbf{v}}}+\nabla \overline{p}=& ~%
{\mathbf{v}}^{\ast }~~\text{ in }~~\mathcal{O} \\
\left( \sigma ({\overline{{\mathbf{v}}}})\mathbf{n}-p\mathbf{n}\right) \cdot 
\boldsymbol{\tau }=& ~0~~\text{ on }~\partial \mathcal{O} \\
\overline{{\mathbf{v}}}\cdot \mathbf{n}=& ~0~~\text{ on }~~S \\
\overline{{\mathbf{v}}}\cdot \mathbf{n}=& ~0~~\text{ on }~~\Omega .
\label{insert1}
\end{align}%
Therewith: if we multiply the structural PDE component (\ref{staticsys5}%
)---in solution variables $(p,\mathbf{v},w_{1},w_{2})$---by given $z\in
H_{0}^{2}(\Omega ),$ with associated fluid-pressure solution $(p(z),{\mathbf{%
v}}(z))\,$of (\ref{huh})--(\ref{givensys2}), integrate by parts, and utilize
the boundary conditions in the BVP (\ref{huh})--(\ref{givensys2}), we then
have: 
\begin{align}
\langle w_{2}^{\ast }+\xi w_{1}^{\ast },z\rangle _{\Omega }=& ~\xi
^{2}\langle w_{1},z\rangle _{\Omega }+\langle \Delta w_{1},\Delta z\rangle
_{\Omega }  \notag \\
& +\langle \sigma ({\mathbf{v}})\mathbf{n}-p\mathbf{n},{\mathbf{v}}%
(z)\rangle _{\partial \mathcal{O}}  \notag \\
& \text{(after using (\ref{crucial})--(\ref{crucial3}) of Lemma \ref%
{staticwellp})}  \notag \\
=& ~\xi ^{2}\langle w_{1},z\rangle _{\Omega }+\langle \Delta w_{1},\Delta
z\rangle _{\Omega }+(\text{div}(\sigma ({\mathbf{v}}))-\nabla p,{\mathbf{v}}%
(z))_{\mathcal{O}}  \notag \\
& +(\sigma ({\mathbf{v}}),\epsilon ({\mathbf{v}}(z))_{\mathcal{O}}-(p,\text{%
div}({\mathbf{v}}(z)))_{\mathcal{O}}  \notag \\
=& ~\xi ^{2}\langle w_{1},z\rangle _{\Omega }+\langle \Delta w_{1},\Delta
z\rangle _{\Omega }+\xi ({\mathbf{v}},{\mathbf{v}}(z))_{\mathcal{O}}+(%
\mathbf{U}\cdot \nabla {\mathbf{v}},\mathbf{v}(z))_{\mathcal{O}}  \notag \\
& +(\sigma ({\mathbf{v}}),\epsilon ({\mathbf{v}}(z)))_{\mathcal{O}}+\frac{1}{%
2}(\text{div}(\mathbf{U}){\mathbf{v}},{\mathbf{v}}(z))_{\mathcal{O}}  \notag
\\
& +\eta ({\mathbf{v}},{\mathbf{v}}(z))_{\mathcal{O}}-(p,\text{div}({\mathbf{v%
}}(z)))_{\mathcal{O}}-({\mathbf{v}}^{\ast },{\mathbf{v}}(z))_{\mathcal{O}}.
\label{almost}
\end{align}%
Now, using the first resolvent relation in (\ref{staticsys3.5}) and invoking
the respective solution maps for (\ref{huh})--(\ref{givensys2}) and (\ref%
{insert})--(\ref{insert1}), we may express the (prospective) solution
component $(p,{\mathbf{v}})$ of (\ref{statice1}) as 
\begin{align}
{\mathbf{v}}& ={\mathbf{v}}(\xi w_{1}-w_{1}^{\ast };p^{\ast };\mathbf{v}%
^{\ast })={\mathbf{v}}(\xi w_{1}-w_{1}^{\ast })~+\overline{{\mathbf{v}}}
\label{r1} \\
p& =~p(\xi w_{1}-w_{1}^{\ast };p^{\ast };\mathbf{v}^{\ast })=p(\xi
w_{1}-w_{1}^{\ast })+\overline{p}  \label{r2}
\end{align}%
(cf. (\ref{d1})--(\ref{d2})). With (\ref{almost}) and (\ref{r1})--(\ref{r2})
in mind: Accordingly, if we define an operator $B\in \mathcal{L}%
(H_{0}^{2}(\Omega ),H^{-2}(\Omega ))$ as 
\begin{align}
\big(B(w),z\big)\equiv & ~\xi ^{2}\langle w,z\rangle _{\Omega }+\langle
\Delta w,\Delta z\rangle _{\Omega }  \notag \\
& +\xi ^{2}({\mathbf{v}}(w),{\mathbf{v}}(z))_{\mathcal{O}}+\xi (\mathbf{U}%
\cdot \nabla {\mathbf{v}}(w),{\mathbf{v}}(z))_{\mathcal{O}}+\frac{\xi }{2}%
\big(\text{div}(\mathbf{U}){\mathbf{v}}(w),{\mathbf{v}}(z)\big)_{\mathcal{O}}
\notag \\
& +\big(\xi \sigma ({\mathbf{v}}(w)),\epsilon ({\mathbf{v}}(z))\big)_{%
\mathcal{O}}+\eta \xi \big({\mathbf{v}}(w),{\mathbf{v}}(z))_{\mathcal{O}%
}-\xi (p(w),\text{div}({\mathbf{v}}(z)))_{\mathcal{O}},  \label{AA}
\end{align}%
-- where $(p(w),{\mathbf{v}}(w))$ solves (\ref{huh})--(\ref{givensys2}) with 
$H_{0}^{2}(\Omega )$ boundary data $w$ -- then finding solution $w_{1}\in
H_{0}^{2}(\Omega )$ of the structural PDE component (\ref{staticsys5})--(\ref%
{staticsys6}) is tantamount to finding solution $w_{1}\in H_{0}^{2}(\Omega )$
of the variational equation 
\begin{equation}
\left\langle B(w_{1}),z\right\rangle _{(H^{-2}\times H_{0}^{2})(\Omega )}=%
\mathcal{F}(z)\text{, for all }z\in H_{0}^{2}(\Omega )\text{;}  \label{vareq}
\end{equation}%
where the functional $\mathcal{F}\in H^{-2}(\Omega )$ is given by 
\begin{align*}
\mathcal{F}(z)\equiv & ~\big({\mathbf{v}}^{\ast },{\mathbf{v}}(z)\big)_{%
\mathcal{O}}+\langle w_{2}^{\ast }+\xi w_{1}^{\ast },z\rangle _{\Omega } \\
& +\xi ({\mathbf{v}}(w_{1}^{\ast }),{\mathbf{v}}(z))_{\mathcal{O}}-\xi (%
\overline{{\mathbf{v}}},{\mathbf{v}}(z))_{\mathcal{O}} \\
& +(\mathbf{U}\cdot \nabla {\mathbf{v}}(w_{1}^{\ast }),{\mathbf{v}}(z))_{%
\mathcal{O}}-(\mathbf{U}\cdot \nabla \overline{{\mathbf{v}}},{\mathbf{v}}%
(z))_{\mathcal{O}} \\
& +\frac{1}{2}(\text{div}(\mathbf{U}){\mathbf{v}}(w_{1}^{\ast }),{\mathbf{v}}%
(z))_{\mathcal{O}}-\frac{1}{2}(\text{div}(\mathbf{U})\overline{{\mathbf{v}}},%
{\mathbf{v}}(z))_{\mathcal{O}} \\
& +\eta ({\mathbf{v}}(w_{1}^{\ast }),{\mathbf{v}}(z))_{\mathcal{O}}-\eta (%
\overline{{\mathbf{v}}},{\mathbf{v}}(z))_{\mathcal{O}} \\
& +(\sigma ({\mathbf{v}}(w_{1}^{\ast })),\epsilon ({\mathbf{v}}(z)))_{%
\mathcal{O}}-(\sigma (\overline{{\mathbf{v}}}),\epsilon ({\mathbf{v}}(z)))_{%
\mathcal{O}} \\
& -(p(w_{1}^{\ast }),\text{div}({\mathbf{v}}(z)))_{\mathcal{O}}+(\overline{p}%
,\text{div}({\mathbf{v}}(z)))_{\mathcal{O}}.
\end{align*}

Recall that by Lemma \ref{staticwellp}(iv), one has the following estimate
for the pressure term in (\ref{AA}), for $\xi =\xi (\mathbf{U})$ large
enough: 
\begin{equation*}
\left\Vert p(w)\right\Vert_{L^2(\mathcal{O})} \leq \frac{C}{\xi }\left\Vert
w\right\Vert _{H_{0}^{2}(\Omega )}\text{, \ \ for all }w\in H_{0}^{2}(\Omega
).
\end{equation*}%
By means of this estimate and Korn's inequality, we will have, in a manner
analogous to that in the proof of Lemma \ref{staticwellp}, that the operator 
$B$ is $H_{0}^{2}(\Omega )$-elliptic, for $\xi =\xi (\mathbf{U})$ large
enough. Thus we can use the Lax-Milgram Theorem to solve the variational
equation (\ref{vareq}), or what is the same, recover the solution component $%
w_{1}$ of the resolvent equations (\ref{staticsys5})--(\ref{staticsys6}). In
turn, we will have 
\begin{equation*}
\begin{array}{c}
w_{2}=~\xi w_{1}-w_{1}^{\ast }, \\ 
{\mathbf{v}}=~{\mathbf{v}}(w_{2})+\overline{{\mathbf{v}}}, \\ 
p=~p(w_{2})+\overline{p}.%
\end{array}%
\end{equation*}%
where again $[p(w_{2}),{\mathbf{v}}(w_{2})]$ is the solution to (\ref{huh}%
)--(\ref{givensys2}), and $[\overline p,\overline{{\mathbf{v}}}]$ solves the
system (\ref{insert})--(\ref{insert1}).

This finally establishes the range condition in (\ref{range_0}) for $\xi >0$
sufficiently large. A subsequent application of Lumer-Philips Theorem yields
a contraction semigroup for the $\widehat{\mathcal{A}}:D(\mathcal{A})\subset 
\mathcal{H}\rightarrow \mathcal{H}$. As a consequence, the application of
Theorem 1.1 \cite[Chapter 3.1]{pazy}, p.76, gives the desired result for the
(unperturbed) compressible flow-structure generator $\mathcal{A}$.

\section{Stationary Problem}

\label{static}

Since the stationary problem associated with a dissipative dynamical system
is of interest when studying long-time behavior of solutions \cite{springer}%
, we discuss the linear stationary problem associated with \eqref{system1}--%
\eqref{IC_2}. We briefly discuss the inclusion of nonlinearity in the plate
for the stationary problem in Section \ref{nonlinear}. (Such a discussion is
in line with \cite[p.658]{Chu2013-comp}.) Formally, we introduce the
following problem: 
\begin{align}
& \left\{ 
\begin{array}{l}
\mathbf{U}\cdot \nabla p+\text{div}~\mathbf{u}=0~\text{ in }~\mathcal{O}%
\times (0,\infty ) \\ 
\mathbf{U}\cdot \nabla \mathbf{u}-\text{div}~\sigma (\mathbf{u})+\eta 
\mathbf{u}+\nabla p=0~\text{ in }~\mathcal{O}\times (0,\infty ) \\ 
(\sigma (\mathbf{u})\mathbf{n}-p\mathbf{n})\cdot \boldsymbol{\tau }=0~\text{
on }~\partial \mathcal{O}\times (0,\infty ) \\ 
\mathbf{u}\cdot \mathbf{n}=0~\text{ on }~\partial \mathcal{O}\times
(0,\infty )%
\end{array}%
\right.   \label{system1*} \\
&  \notag \\
& \left\{ 
\begin{array}{l}
\Delta ^{2}w+\left[ 2\nu \partial _{x_{3}}(\mathbf{u})_{3}+\lambda \text{div}%
(\mathbf{u})-p\right] _{\Omega }=0~\text{ on }~\Omega \times (0,\infty ) \\ 
w=\frac{\partial w}{\partial \nu }=0~\text{ on }~\partial \Omega \times
(0,\infty )%
\end{array}%
\right.   \label{IM2*}
\end{align}%
(We note here, as in Remark \ref{weaker}, the boundary condition~ $(\sigma (%
\mathbf{u})\mathbf{n}-p\mathbf{n})\cdot \boldsymbol{\tau }=0$---for $%
\mathbf{\tau }\in TH^{1/2}(\partial \mathcal{O})$---is to be interpreted in
the sense of distributions.)

Note that in terms of the fluid-structure generator $\mathcal{A}:D(\mathcal{A%
})\subset \mathcal{H}\rightarrow \mathcal{H}$, solving the PDE system (\ref%
{system1*})--(\ref{IM2*}) is equivalent to identifying an element in $Null(%
\mathcal{A})$. Alternatively, as in \cite{Chu2013-comp}, the problem (\ref%
{system1*})--(\ref{IM2*}) can be interpreted variationally. This is to say,
we define a \emph{weak solution} to \eqref{system1*}--\eqref{IM2*} to be a
triple $(p,\mathbf{u},w)\in L^{2}(\mathcal{O})\times \mathbf{V}_{0}\times
H_{0}^{2}(\Omega )$, which must satisfy, 
\begin{align}
(\text{div}~\mathbf{u},q)_{\mathcal{O}}& =~\int_{\mathcal{O}}(\text{div}~%
\mathbf{U})(pq)dx+\int_{\mathcal{O}}(\mathbf{U}\cdot \nabla q)pdx
\label{variationally1} \\
a_{\mathcal{O}}(\mathbf{u},\boldsymbol{\psi })-(p,\text{div}~\boldsymbol{%
\psi })_{\mathcal{O}}=& ~-\int_{\mathcal{O}}(\mathbf{U}\cdot \nabla \mathbf{u%
})\boldsymbol{\psi }dx-(\Delta w,\Delta \beta )_{\Omega },
\label{variationally2}
\end{align}%
where the bilinear form $a_{\mathcal{O}}(\cdot ,\cdot )$ is as given in (%
\ref{bi}) for all $q\in W_{\mathbf{U}}$ and all $\boldsymbol{\psi }\in 
\mathbf{V}_{\Omega }$. We take:%
\begin{equation}
W_{\mathbf{U}}=\left\{ q\in L^{2}(\mathcal{O}):\mathbf{U}\cdot \nabla q\in
L^{2}(\mathcal{O})\right\} .  \label{W_U}
\end{equation}%
\begin{equation}
\mathbf{V}_{\Omega }\equiv \left\{ \mathbf{v=v(}\beta )\in \mathbf{H}^{1}(%
\mathcal{O})~:~\left[ \mathbf{v}\cdot \mathbf{n}\right] _{\partial \mathcal{O%
}}=\begin{cases}
0~\text{ on }~S \\ 
\beta \in H_{0}^{2}(\Omega )\text{ \ on \ }\Omega .%
\end{cases}.  \label{V_omega} \right\}
\end{equation}%
Note that by Theorem 5 of \cite{buffa}, and extension by zero of $\beta \in
H_{0}^{2}(\Omega )$, the space $\mathbf{V}_{\Omega }$ is well-defined on the
Lipschitz geometry of $\mathcal{O}.$

We note that (\ref{variationally1})--(\ref{variationally2}) is a \emph{natural%
} definition of a weak (variational) solution, in line with \emph{weak}
solutions \cite{Chu2013-comp} to the dynamic equations \eqref{system1}--%
\eqref{IC_2}. Indeed, we demonstrate that such weak solutions---{\bf should they exist}---are classical
solutions of the PDE system \eqref{system1*}--(\ref{IM2*}).

\begin{lemma}
\label{mark}Suppose that $(p,\mathbf{u},w)\in L^{2}(\mathcal{O})\times 
\mathbf{V}_{0}\times H_{0}^{2}(\Omega )$ and satisfies \eqref{variationally1}--\eqref{variationally2}. Then $(p,\mathbf{u},w)$ satisfies \eqref{system1*}--\eqref{IM2*} almost everywhere, and in fact $\left[ p,\mathbf{u},w,0\right]
\in Null(\mathcal{A})$.
\end{lemma}

\begin{proof} Here, we essentially mimic the final part of the proof of
Lemma 5.1.

Firstly, in (\ref{variationally1}) we consider $q\in \mathcal{D}(\mathcal{O})$.
An invocation of Green's Theorem then yields
\begin{eqnarray*}
(\func{div}(\mathbf{u}),q)_{\mathcal{O}} &=&~\int_{\mathcal{O}}\func{div}(%
\mathbf{U})(pq)dx+0-\left( \int_{\mathcal{O}}\func{div}(\mathbf{U}%
)(qp)dx+\int_{\mathcal{O}}(\mathbf{U}\cdot \nabla p)qdx\right)  \\
&=&-\int_{\mathcal{O}}(\mathbf{U}\cdot \nabla p)qdx~~\text{ \ for every }q\in 
\mathcal{D}(\mathcal{O}).
\end{eqnarray*}%
The density of $\mathcal{D}(\mathcal{O})$ in $L^{2}(\mathcal{O})$ then
yields that 
\begin{equation}
\mathbf{U}\cdot \nabla p+\func{div}(\mathbf{u})=0\text{ \ in the } L^{2}(%
\mathcal{O})\text{-sense}  \label{r_1}
\end{equation}%
(and so, with $\mathbf u \in \mathbf V_0$, in particular, $p\in W_{\mathbf{U}}$ of (\ref{W_U})).

Secondly, if in (\ref{variationally2}) we take test function $\mathbf{\psi }%
\in \left[ \mathcal{D}(\mathcal{O})\right] ^{3}$, then we infer, upon
integration by parts, that 
\begin{equation}
-\func{div}\sigma (\mathbf{u})+\eta \bu+\nabla p+\mathbf{U}\cdot \nabla 
\mathbf{u}=0\text{ \ in }\mathbf{L}^{2}(\mathcal{O})\text{-sense.}
\label{r_2}
\end{equation}%
Subsequently, from (\ref{r_2}), the fact that $\left\{ p,\mathbf{u}\right\}
\in L^{2}(\mathcal{O})\times \mathbf{H}^{1}(\mathcal{O)}$, and by an
integration by parts, we can (as before) assign a meaning to the boundary trace term~  $%
\left[ \sigma (\mathbf{u)n}-p\mathbf{n}\right] _{\partial \mathcal{O}}$,
viz.,%
\begin{equation}
\sigma (\mathbf{u)n}-p\mathbf{n\in H}^{-\frac{1}{2}}(\partial \mathcal{O}).
\label{r_3}
\end{equation}
Applying this boundary trace to the relation (\ref{variationally2}), we then
have for a test function $\mathbf{\psi }\in \mathbf{V}_{0}\subset \mathbf{V}%
_{\Omega }$, upon an integration by parts (and considering (\ref{r_2})) we
make the inference
\begin{equation}
\left\langle \sigma (\mathbf{u)n}-p\mathbf{n},\mathbf{\psi }\right\rangle
_{\partial \mathcal{O}}=0\text{ \ for every }\mathbf{\psi }\in \mathbf{V}_{0}%
\text{,}  \label{r_3.5}
\end{equation}%
Thus, the following tangential boundary condition is satisified:%
\begin{equation}
\left[ \sigma (\mathbf{u)n}-p\mathbf{n}\right] \cdot \mathbf{\tau }=0\text{
\ for all }\mathbf{\tau }\in TH^{1/2}(\partial \mathcal{O}).\footnote{See again Remark \ref{weaker}.}  \label{r_4}
\end{equation}%

Thirdly, with respect to (\ref{variationally2}), we have upon integration
by parts with variational relation, and an invocation of (\ref{r_2}),%
\begin{equation}
\left\langle \sigma (\mathbf{u)n}-p\mathbf{n},\mathbf{\psi }\right\rangle
_{\partial \mathcal{O}}=-(\Delta w,\Delta \beta )_{\Omega }\text{ \ for
every }\mathbf{\psi }\in \mathbf{V}_{\Omega }.  \label{r_5}
\end{equation}%
In particular, if for given $\beta \in H_{0}^{2}(\Omega )$, we set 
\begin{equation}
\mathbf{\psi }=\left\{ 
\begin{array}{c}
0\text{ \ in }S\text{,} \\ 
\beta \mathbf{n}\text{ \ in }\Omega 
\end{array}%
\right.,  \label{r_6}
\end{equation}%
then this $\mathbf{\psi }\in \mathbf{V}_{\Omega }$ (see e.g., Theorem 3.33
of \cite{Mc}). Applying this test function in (\ref{r_6}) to (\ref{r_5}) --
and using $\left. \mathbf{n}\right\vert _{\Omega }=\left[ 0,0,1\right] $---we have 
\begin{equation*}
\left( \left[ 2\nu \partial _{x_{3}}(\mathbf{u})_{3}+\lambda \text{div}(%
\mathbf{u})-p\right] _{\Omega },\beta \right) _{\Omega }=-(\Delta w,\Delta
\beta )_{\Omega }\text{ \ for every }\mathbf{\beta }\in H_{0}^{2}(\Omega ).
\end{equation*}%
In particular, this holds for ${\beta }\in \mathcal{D}(\Omega )$,
whence we obtain%
\begin{equation}
\Delta ^{2}w+\left[ 2\nu \partial _{x_{3}}(\mathbf{u})_{3}+\lambda \text{div}%
(\mathbf{u})-p\right] =0\text{ in }~L^{2}(\Omega )\text{-sense.}
\label{r_7}
\end{equation}

Upon collecting (\ref{r_1}), (\ref{r_2}), (\ref{r_4}) and (\ref{r_7}), we
then have that variational solution $\left[ p,\mathbf{u},w\right] $ of (\ref%
{variationally1})--(\ref{variationally2}) solves the coupled PDE system (\ref%
{system1*})--(\ref{IM2*}) a.e. (Note that since solution component $\mathbf{u}%
\in \mathbf{V}_{0}$, then its normal component on $\partial \mathcal{O}$ is
zero.)\end{proof}

We note that for any $\boldsymbol{\psi }$ and $\beta $ as above, we can
utilize the weak identities to see that a stationary solution of (\ref%
{variationally1})--(\ref{variationally2}) must satisfy, for all $q\in W_{%
\mathbf{U}}$ and $\mathbf{\psi }\in \mathbf{V}_{\Omega }$, 
\begin{align}
a_{\mathcal{O}}(\mathbf{u},\boldsymbol{\psi })-(p,\text{div}~\boldsymbol{%
\psi })_{\mathcal{O}}+& \big(\Delta w,\Delta \beta \big)_{\Omega }+(\text{div%
}~\mathbf{u},q)  \label{tester} \\
=& -\int_{\mathcal{O}}(\mathbf{U}\cdot \nabla \mathbf{u})\boldsymbol{\psi }%
dx+~\int_{\mathcal{O}}(\text{div}~\mathbf{U})(pq)dx+\int_{\mathcal{O}}(%
\mathbf{U}\cdot \nabla q)pdx.  \notag
\end{align}%
Since $p\in W_{\mathbf{U}}$ (as we saw in the proof of Lemma \ref{mark}), we
may choose $q=p$ and $\boldsymbol{\psi }=\mathbf{u}\in \mathbf{V}_{0}$ in (%
\ref{tester}), invoke Green's theorem and the divergence theorem as above,
to see that 
\begin{equation}
a_{\mathcal{O}}(\mathbf{u},\mathbf{u})=\frac{1}{2}\int_{\mathcal{O}}(\text{%
div}~\mathbf{U})[|p|^{2}+|\mathbf{u}|^{2}]dx.  \label{zero}
\end{equation}

As discussed above, it is clear that the long-time behavior (stability)
properties of solutions to the dynamics \eqref{system1*}--\eqref{IM2*}
depend on the structure of the flow field $\mathbf{U}$. Thus, from (\ref%
{zero}) it is not an unwelcome assumption (see \cite{preprint}) to consider
a divergence-free flow field $\mathbf{U}$. This yields the following theorem.

\begin{lemma}
\label{weakchar} If the ambient vector field $\mathbf{U}$ is divergence free, a
weak solution to (\ref{variationally1})-(\ref{variationally2}) (or
equivalently, a solution to \eqref{system1*}--\eqref{IM2*}) will be a triple
of the form $(c,\mathbf{0},w_{c})$, with $c=$~const, and where $w_{c}\in
H_{0}^{2}(\Omega )$ solves the boundary value  problem 
\begin{equation}
\Delta ^{2}w=c~~\text{ in }~~\Omega \text{, \ }w=\nabla w=0\text{ \ on }%
\partial \Omega \text{.}  \label{bvp_2}
\end{equation}
\end{lemma}

\begin{proof} If div$(\mathbf{U})=0$, then from (\ref{zero}) and Korn's
Inequality we have that $\mathbf{u}=0$. Subsequently, the fluid equation in (%
\ref{system1*}) gives $\nabla p=0$, and so $p=c$. In turn, the structural
equation in (\ref{IM2*}) with such $\left\{ \mathbf{u},p\right\} $ becomes
(\ref{bvp_2}). \end{proof}

\section{Plate Nonlinearity}

\label{nonlinear}

The treatment of semilinear, cubic-type nonlinearities in fluid-plate
problems has become popular (see the surveys \cite{berlin11,dcds} and
references therein). In this section we demonstrate well-posedness of mild
solutions to the dynamic problem, as well as discuss the stationary problem, 
\emph{in the presence of the scalar von Karman nonlinearity} \cite{springer}%
. We begin with some basic facts about the von Karman nonlinearity,
introduced in \eqref{IM2**}--\eqref{airy-1}. The first of which revolves
around the \emph{local Lipschitz} property of $f$ from $H_{0}^{2}(\Omega
)\rightarrow L^{2}(\Omega )$. By way of availing ourselves of said Lipschitz
continuity for von Karman plates, we further assume that bounded $\Omega
\subset \mathbb{R}^{2}$ is sufficiently smooth.

This property relies on the so called sharp regularity of the Airy stress
function: Corollary 1.4.4 in \cite{springer}. \ To begin, one has the estimate,%
\begin{equation}
\Vert \left( \Delta _{D}^{2}\right) ^{-1}[u,w]\Vert _{W^{2,\infty }(\Omega
)}\leq C\Vert u\Vert _{2,\Omega }\Vert w\Vert _{2,\Omega },  \label{bi_D}
\end{equation}%
where $\Delta _{D}^{2}$ denotes the biharmonic operator with clamped
boundary conditions. With $v(w)=v(w,w)$ in (\ref{airy-1}), the estimate (%
\ref{bi_D}) above yields 
\begin{equation*}
\Vert v(w)\Vert _{W^{2,\infty }(\Omega )}\leq C\Vert w\Vert _{2,\Omega }^{2},
\end{equation*}%
which, in turn, implies that the Airy stress function $v(w)$ satisfies the inequality 
\begin{equation}
\Vert \lbrack u_{1},v(u_{1})]-[u_{2},v(u_{2})]\Vert _{L^{2}(\Omega )}\leq C%
\Big(\Vert u_{1}\Vert _{H_{0}^{2}(\Omega )}^{2}+\Vert u_{2}\Vert
_{H_{0}^{2}(\Omega )}\Vert _{H_{0}^{2}(\Omega )}^{2}\Big)\Vert
u_{1}-u_{2}\Vert _{H_{0}^{2}(\Omega )}  \label{airy-lip}
\end{equation}%
(see Corollary 1.4.5 in \cite{springer}). Thus, the nonlinearity $%
f(w)=[w,v(w)+F_{0}]$ is locally Lipschitz from $H_{0}^{2}({\Omega })$ \i nto 
$L^{2}(\Omega )$.

The second critical property of the nonlinearity involves the existence of a
potential energy functional $\Pi$ associated with $f$. In the case of the
von Karman nonlinearity, it has the form 
\begin{equation*}
\Pi(w)=\frac14\int_{\Omega}\Big(|\Delta v(w)|^2 -2w[w,F_0] \Big) dx,
\end{equation*}
and possesses the properties that $\Pi$ is a $C^1$-functional on $%
H^2_0(\Omega)$ such that $f$ is a Fr\'echet derivative of $\Pi$: $%
-f(w)=\Pi^{\prime }(w)$. From this it follows that for a smooth function $w$%
: 
\begin{equation*}
\dfrac{d}{dt}\Pi(w)=(\Pi^{\prime }(w),w_t)=-(f(w),w_t)_{\Omega}.
\end{equation*}
Moreover $\Pi(\cdot)$ is locally bounded on $H^2_0(\Omega)$, and there exist 
$\eta<1/2$ and $C\ge 0$ such that 
\begin{equation}  \label{8.1.1c1}
\eta \|\Delta w\|_{\Omega}^2 +\Pi(w)+C \ge 0\;,\quad \forall\, w\in
H^2_0(\Omega)\;.
\end{equation}
The latter fact follows from the bound \cite[Chapter 1.4]{springer} 
\begin{equation*}
||w||^2_{\theta} \le \epsilon\big[ ||\Delta w||^2+||\Delta v(w)||^2\big]%
+C_{\epsilon},~~\theta \in [0,2).
\end{equation*}

\begin{remark}
We note that the Berger and Kirchhoff nonlinearities, for instance discussed
in \cite{supersonic,Chu2013-comp}, satisfy the above properties; they:
(i) are locally Lipschitz $H_0^2(\Omega)\to L^2(\Omega)$, (ii) have a $C^1$
antiderivative $\Pi$ satisfying the above properties.
\end{remark}

\subsection{Nonlinear Dynamic Problem}

We now address the system \eqref{system1}--\ref{IC_2}, taken with act\i ve
plate nonlinearity: 
\begin{equation*}
w_{tt}+\Delta ^{2}w+\left[ 2\nu \partial _{x_{3}}(\mathbf{u})_{3}+\lambda 
\text{div}(\mathbf{u})-p\right] _{\Omega }=[w,v(w)+F_{0}]~\text{ on }~\Omega
\times (0,\infty ).
\end{equation*}%
We will show the well-posedness of \emph{mild solutions} (in the sense of 
\cite{pazy}) in the presence of the von Karman nonlinearity. To th\i s end,
we define a nonlinear operator $\mathcal{F}:\mathcal{H}\rightarrow \mathcal{H%
}$, given by 
\begin{equation*}
\mathcal{F}\big(\lbrack p,\mathbf{u},w_{1},w_{2}]\big)=[0,\mathbf{0}%
,0,f(w_{1})].
\end{equation*}%
This mapping is locally Lipschitz (by the properties of $f$ above), and thus
will be considered as a perturbation to the l\i near flu\i d-structure
Cauchy problem wh\i ch \i s modelled by generator $\mathcal{A}:\mathcal{H}%
\rightarrow \mathcal{H}$ In part\i cular we have the abstract problem \i n
var\i able $\mathbf{y}\in \mathcal{H}$,%
\begin{align}
\mathbf{y}^{\prime }=& ~\mathcal{A}\mathbf{y}+\mathcal{F}(\mathbf{y})
\label{Cauchy2} \\
\mathbf{y}(0)=& ~\mathbf{y}_{0}\in \mathcal{H}.  \label{Cauchy3}
\end{align}

\begin{theorem}
\label{th:nonlin} The nonlinear Cauchy problem in \eqref{Cauchy2}--%
\eqref{Cauchy3} is well-posed in the sense of mild solutions. This is to
say: there is a unique local-in-time mild solution $\mathbf{y}(t)$ on $t\in
\lbrack 0,t_{\text{max}})$ (which is also a weak solut\i on). Moreover, for $%
\mathbf{y}_{0}\in {D}(\mathcal{A})$, the corresponding solution is
strong.

In e\i ther case, when $t_{\text{max}}(\mathbf{y}_{0})<\infty $, we have
that $||\mathbf{y}(t)||_{\mathcal{H}}\rightarrow \infty $ as $t\nearrow t_{%
\text{max}}(y_{0})$.
\end{theorem}

\begin{proof}
With the fluid-structure semigroup $e^{\mathcal At}$ in hand from Theorem \ref{wellp}, this is a direct application of Theorem 1.4~\cite[p.185]{pazy}
and  localized version of Theorem 1.6~\cite[p.189]{pazy}, pertaining to locally Lipschitz perturbations of semigroup solutions.  
\end{proof}In order to guarantee global solutions, i.e., valid solutions of (%
\ref{Cauchy2})--(\ref{Cauchy3}) on $[0,T]$ for any $T>0$, we must utilize the
\textquotedblleft good" structure of $\Pi $. The energy identity, in the
presence of nonlinearity (i.e., when \eqref{IM2} has the term $f(w)$), is
obtained in a standard way using the properties of $\Pi $ (see \cite%
{supersonic,Chu2013-comp}). Consider $\mathbf{y}_{0}=(p_{0},\mathbf{u}%
_{0},w_{0},w_{1})\in \mathcal{H}$ and $\mathbf{U}\in \mathbf{V}_{0}$. Any
mild solution corresponding to (\ref{Cauchy2})--(\ref{Cauchy3}) satisfies: 
\begin{align}
\mathcal{E}\big(p(t),\mathbf{u}(t),w(t),w_{t}(t)\big)+\int_{0}^{t}a_{%
\mathcal{O}}(\mathbf{u}(\tau ),\mathbf{u}(\tau ))d\tau +\Pi (w(t))=& ~%
\mathcal{E}\big(p_{0},\mathbf{u}_{0},w_{0},w_{1}\big)+\Pi (w(0))
\label{balancelaw*} \\
& +\frac{1}{2}\int_{0}^{t}\int_{\mathcal{O}}\text{div}(\mathbf{U})[|p(\tau
)|^{2}+|\mathbf{u}(\tau )|^{2}]dxd\tau .  \notag
\end{align}%
From this a priori relation, Gronwall's inequality, and the bound on the $%
\Pi$ in \eqref{8.1.1c1}, we have the final (nonlinear) well-posedness theorem.

\begin{theorem}
\label{final} For any $T>0$, the Cauchy problem (\ref{Cauchy2})-(\ref%
{Cauchy3}) is well-posed on $\mathcal{H}$ for all $[0,T]$. This is to say
that the PDE problem in \eqref{system1}--\eqref{IC_2}, taking into account
the nonlinear plate equation \eqref{IM2**}, is well-posed in the sense of
mild solutions.

Moreover, in the case of $\text{div}~\mathbf{U }\equiv 0$, we have the \emph{%
global-in-time} estimate for solutions: 
\begin{equation}  \label{globalbound}
\sup_{t \in [0,\infty)} \mathcal{E}(p(t),\mathbf{u}(t),w(t),w_t(t)) \le 
\mathbf{C}(p_0,\mathbf{u}_0,w_0,w_1, F_0).
\end{equation}
\end{theorem}

\begin{proof}
The proof follows a standard tack, and is along the lines of \cite{Chu2013-comp} (in the case of these dynamics, taken with $\mathbf U=0$), or \cite{supersonic,webster} (for the case of compressible, inviscid gas dynamics). See also \cite[Chapter 2.3]{springer} for an abstract discussion of nonlinear second order evolutions with locally Lipschitz perturbations.
\end{proof}

\subsection{Nonlinear Stationary Problem}

We now briefly mention the nonlinear stationary problem in the case when $%
\text{div}~\mathbf{U }\equiv 0$. As noted above (and, as is evident from %
\eqref{balancelaw*} and \eqref{globalbound}), this is the primary case of
interest for \emph{long time behavior} analysis of the dynamics %
\eqref{system1}--\eqref{IM2}.

We note that the analysis above in Section \ref{static} obtains identically
in the presence of plate nonlinearity. Thus, with $\mathbf{U}$ divergence
free as above, we have the equivalence of weak solutions to the system 
\begin{align}
& \left\{ 
\begin{array}{l}
\mathbf{U}\cdot \nabla p+\text{div}~\mathbf{u}=0~\text{ in }~\mathcal{O}%
\times (0,\infty ) \\ 
\mathbf{U}\cdot \nabla \mathbf{u}-\text{div}~\sigma (\mathbf{u})+\eta 
\mathbf{u}+\nabla p=0~\text{ in }~\mathcal{O}\times (0,\infty ) \\ 
(\sigma (\mathbf{u})\mathbf{n}-p\mathbf{n})\cdot \boldsymbol{\tau }=0~\text{
on }~\partial \mathcal{O}\times (0,\infty ) \\ 
\mathbf{u}\cdot \mathbf{n}=0~\text{ on }~S\times (0,\infty ) \\ 
\mathbf{u}\cdot \mathbf{n}=0~\text{ on }~\Omega \times (0,\infty )%
\end{array}%
\right.  \label{system1***} \\
&  \notag \\
& \left\{ 
\begin{array}{l}
\Delta ^{2}w+\left[ 2\nu \partial _{x_{3}}(\mathbf{u})_{3}+\lambda \text{div}%
(\mathbf{u})-p\right] _{\Omega }=[w,v(w)+F_{0}]~\text{ on }~\Omega \times
(0,\infty ) \\ 
w=\frac{\partial w}{\partial \nu }=0~\text{ on }~\partial \Omega \times
(0,\infty )%
\end{array}%
\right.  \label{IM2***}
\end{align}%
and the following biharmonic problem: Find $w\in H_{0}^{2}(\Omega )$ such that 
\begin{equation}
(\Delta w,\Delta v)_{\Omega }-(f(w),v)=(c,v)_{\Omega },~~\forall ~~v\in
H_{0}^{2}(\Omega ).  \label{nonstat}
\end{equation}%
(Note, this reduction is equivalent to that in \cite[Section 3.3]%
{Chu2013-comp}.) With property \eqref{8.1.1c1} of $\Pi $, it is well known
that (for a given $c$) the solutions to \eqref{nonstat}, denoted by $%
\mathcal{N}_{c}$, form a nonempty, compact set in $H_{0}^{2}(\Omega )$ \cite%
{Chu2013-comp,ChuRyz2011,springer}\footnote{%
The structure of $\mathcal{N}_{c}$ is dependent upon the in-plane forcing $%
F_{0}$}. Th\i s leaves us with the final theorem:

\begin{theorem}
Assume $\text{div}~\mathbf{U }\equiv 0$. Weak solutions to \eqref{system1***}%
--\eqref{IM2***} are fully characterized by points of the form: 
\begin{equation*}
(c,\mathbf{0},w_c),~~w_c \in \mathcal{N}_c \subset \subset H_0^2(\Omega).
\end{equation*}
\end{theorem}

\section{Acknowledgments}

The authors would like to thank both anonymous referees for their careful
reading of the paper, and thoughtful feedback which improved the quality of
the paper.

The first and third author would like to thank the National Science
Foundation, and acknowledge their partial funding from NSF Grants
DMS-1211232 and DMS-1616425 (G. Avalos) and NSF Grant DMS-1504697 (J.T.
Webster).

\section{Appendix}

For the reader's convenience, we provide an explicit proof for the
well-posedness of the (uncoupled) pressure equation (\ref{above}).

\begin{theorem}
\label{dV}(See \cite{dV} and \cite{LaxPhil}) Consider the linear equation 
\begin{equation}
k\lambda +\mathbf{v}\cdot \nabla \lambda +\frac{1}{2}\lambda \func{div}(%
\mathbf{v})=G\text{ \ in }\mathcal{O},  \label{linear}
\end{equation}%
(as before, $\mathcal{O}\subset \mathbb{R}^{3}$ is a bounded, convex
domain). Also, the parameter $k>0$ and forcing term $G\in L^{2}(\mathcal{O})$%
. Moreover, the fixed vector field $\mathbf{v}$ in (\ref{linear}) is in $%
\mathbf{H}^{3}(\mathcal{O})$ and further satisfies:%
\begin{equation}
\begin{array}{l}
\text{(i) }\mathbf{v}\cdot \mathbf{n}=0\text{ \ on }\partial \mathcal{O}%
\text{;} \\ 
\text{(ii) }2\left\Vert \nabla \mathbf{v}\right\Vert _{L^{\infty }(\mathcal{O%
})}+\frac{C_{S}}{2}meas(\mathcal{O})^{\frac{1}{6}}\left\Vert \Delta \func{div%
}(\mathbf{v})\right\Vert _{\mathcal{O}}\leq k,%
\end{array}
\label{v_a}
\end{equation}%
where $C_{S}>0$ is a constant which gives rise to the Sobolev embedding
inequality,%
\begin{equation}
\left\Vert f\right\Vert _{L^{6}(\mathcal{O})}\leq \sqrt{C_{S}}\left\Vert
\nabla f\right\Vert _{\mathcal{O}}\text{ }\ \text{for all }f\in H_{0}^{1}(%
\mathcal{O}).  \label{sob}
\end{equation}%
Then for given $G\in L^{2}(\mathcal{O})$, there exists a unique $\lambda \in
L^{2}(\mathcal{O})$ which is a weak solution of (\ref{linear}). Moreover,
the solution satisfies the bound 
\begin{equation}
\Vert \lambda \Vert _{\mathcal{O}}\leq \dfrac{1}{k}\Vert G\Vert _{\mathcal{O}%
}.  \label{est}
\end{equation}%
By the $L^{2}(\mathcal{O})$-function $\lambda $ being a weak solution of (%
\ref{linear}), we mean that it satisfies the following variational relation:%
\begin{equation}
k\int_{\mathcal{O}}\lambda \varphi d\mathcal{O-}\int_{\mathcal{O}}\lambda (%
\mathbf{v}\cdot \nabla \varphi )d\mathcal{O-}\frac{1}{2}\int_{\mathcal{O}%
}\lambda \func{div}(\mathbf{v})\varphi d\mathcal{O}=\int_{\mathcal{O}%
}G\varphi d\mathcal{O}\text{ \ for every }\varphi \in H^{1}(\mathcal{O}).
\label{var}
\end{equation}
\end{theorem}

\begin{proof}[Proof of Theorem \ref{dV}:] In large part, the present proof is
taken from that of Appendix I of \cite{dV} (on p.541)---see also \cite%
{LaxPhil}---which however was undertaken on the assumption that geometry $%
\mathcal{O}$ is smooth. Accordingly, adjustments are made here for general
convex domain $\mathcal{O}$, as well as for the perturbation $\frac{\lambda 
}{2}\func{div}(\mathbf{v})$.

Given $\epsilon >0$, we denote $\lambda _{\epsilon }$ to be the solution of
the following regularized boundary value problem:%
\begin{equation}
\left\{ 
\begin{array}{l}
-\epsilon \Delta \lambda _{\epsilon }+k\lambda _{\epsilon }+\mathbf{v}\cdot
\nabla \lambda _{\epsilon }+\frac{1}{2}\lambda _{\epsilon }\func{div}(%
\mathbf{v})=G\text{ \ in }\mathcal{O}\text{,} \\ 
\left. \lambda _{\epsilon }\right\vert _{\partial \mathcal{O}}=0.%
\end{array}%
\right.  \label{bvp}
\end{equation}%
Assume initially that data $G\in H_{0}^{1}(\mathcal{O})$. We note that the
Lax-Milgram Theorem insures the existence of the $H_{0}^{1}(\mathcal{O})$%
-function $\lambda _{\epsilon }$ which solves (\ref{bvp}): Indeed,
multiplying the left hand side of the equation in (\ref{bvp}) by solution
variable $\lambda _{\epsilon }$, integrating, and then integrating by parts,
we have 
\begin{eqnarray}
&&\epsilon \left\Vert \nabla \lambda _{\epsilon }\right\Vert _{\mathcal{O}%
}^{2}+\kappa \left\Vert \lambda _{\epsilon }\right\Vert _{\mathcal{O}}^{2}-%
\frac{1}{2}\int_{\mathcal{O}}\func{div}(\mathbf{v})\left\vert \lambda
_{\epsilon }\right\vert ^{2}d\mathcal{O}+\frac{1}{2}\int_{\mathcal{O}}\func{%
div}(\mathbf{v})\left\vert \lambda _{\epsilon }\right\vert ^{2}d\mathcal{O} 
\notag \\
&=&\epsilon \left\Vert \nabla \lambda _{\epsilon }\right\Vert _{\mathcal{O}%
}^{2}+\kappa \left\Vert \lambda _{\epsilon }\right\Vert _{\mathcal{O}}^{2}.
\label{bvp2}
\end{eqnarray}
Thus we infer $H_{0}^{1}(\mathcal{O})$-ellipticity for the bilinear form
associated with the PDE (\ref{bvp}).

Subsequently, since the bounded domain $\mathcal{O}$ is convex, we can invoke
Theorem 3.2.1.2, p.147, of \cite{grisvard} to have that solution $\lambda
_{\epsilon }\in H^{2}(\mathcal{O})\cap H_{0}^{1}(\mathcal{O})$. This extra
regularity allows for a multiplication of both sides of (\ref{bvp}) by term $%
\Delta \lambda _{\epsilon }$, and a subsequent integration and integration
by parts, so as to yield the relation%
\begin{equation}
-\epsilon \left\Vert \Delta \lambda _{\epsilon }\right\Vert _{\mathcal{O}%
}^{2}-\kappa \left\Vert \nabla \lambda _{\epsilon }\right\Vert _{\mathcal{O}%
}^{2}+\int_{\mathcal{O}}\Delta \lambda _{\epsilon }\left( \mathbf{v}\cdot
\nabla \lambda _{\epsilon }\right) d\mathcal{O+}\frac{1}{2}\int_{\mathcal{O}}%
\func{div}(\mathbf{v})\lambda _{\epsilon }\Delta \lambda _{\epsilon }%
\mathcal{=}\int_{\mathcal{O}}G\Delta \lambda _{\epsilon }d\mathcal{O}.
\label{bvp3}
\end{equation}%
To handle the third term on left hand side: Using classic vector field
identities for the wave equation---see e.g., \cite{morawetz}, \cite{chen},
or \cite{trigg} p.459---we have 
\begin{eqnarray}
\int_{\mathcal{O}}\Delta \lambda _{\epsilon }\left( \mathbf{v}\cdot \nabla
\lambda _{\epsilon }\right) d\mathcal{O} &\mathcal{=}&\int_{\partial 
\mathcal{O}}\frac{\partial \lambda _{\epsilon }}{\partial \mathbf{n}}\mathbf{%
v}\cdot \nabla \lambda _{\epsilon }d\partial \mathcal{O-}\int_{\mathcal{O}%
}\left( \nabla \mathbf{v}\nabla \lambda _{\epsilon }\right) \cdot \nabla
\lambda _{\epsilon }d\mathcal{O}  \notag \\
&&+\frac{1}{2}\int_{\mathcal{O}}\left\vert \nabla \lambda _{\epsilon
}\right\vert ^{2}\func{div}(\mathbf{v})d\mathcal{O}.  \label{bvp4}
\end{eqnarray}%
(In stating this relation, we are using assumption (i) of (\ref{v_a}).) With
respect to the first term on right hand side of (\ref{bvp4}): using
Proposition 4, p.702 of \cite{buffa}, and the fact that $\left. \lambda
_{\epsilon }\right\vert _{\partial \mathcal{O}}=0$, we have on $\partial 
\mathcal{O}$ 
\begin{eqnarray*}
\left. \nabla \lambda _{\epsilon }\right\vert _{\partial \mathcal{O}}
&=&\nabla _{\partial \mathcal{O}}(\left. \lambda _{\epsilon }\right\vert
_{\partial \mathcal{O}})+\mathbf{n}\frac{\partial \lambda _{\epsilon }}{%
\partial \mathbf{n}} \\
&=&\mathbf{n}\frac{\partial \lambda _{\epsilon }}{\partial \mathbf{n}}.
\end{eqnarray*}%
(Above, $\nabla _{\partial \mathcal{O}}(\left. \lambda _{\epsilon
}\right\vert _{\partial \mathcal{O}})\in \mathbf{L}^{2}(\partial \mathcal{O)}
$ denotes the tangential gradient of $\left. \lambda _{\substack{ \epsilon 
\\ }}\right\vert _{\partial \mathcal{O}}$; see \cite{necas} or p.701 of 
\cite{buffa}.) Applying this relation to (\ref{bvp4}), and considering $%
\left. \mathbf{v}\cdot \mathbf{n}\right\vert _{\partial \mathcal{O}}=0$, we
then have 
\begin{equation}
\int_{\mathcal{O}}\Delta \lambda _{\epsilon }\left( \mathbf{v}\cdot \nabla
\lambda _{\epsilon }\right) d\mathcal{O=}-\int_{\mathcal{O}}\left( \nabla 
\mathbf{v}\nabla \lambda _{\epsilon }\right) \cdot \nabla \lambda _{\epsilon
}d\mathcal{O}+\frac{1}{2}\int_{\mathcal{O}}\left\vert \nabla \lambda
_{\epsilon }\right\vert ^{2}\func{div}(\mathbf{v})d\mathcal{O}.
\label{bvp4.5}
\end{equation}

\smallskip

For the fourth term on the left hand side of (\ref{bvp3}): Using Green's
Theorem, we have 
\begin{eqnarray}
\left( \lambda _{\epsilon }\func{div}(\mathbf{v}),\Delta \lambda _{\epsilon
}\right) _{\mathcal{O}} &=&-\left( \nabla \lbrack \lambda _{\epsilon }\func{%
div}(\mathbf{v})],\nabla \lambda _{\epsilon }\right) _{\mathcal{O}}  \notag
\\
&=&-\int_{\mathcal{O}}\left\vert \nabla \lambda _{\epsilon }\right\vert ^{2}%
\func{div}(\mathbf{v})d\mathcal{O-}\left( \lambda _{\epsilon }\nabla \lbrack 
\func{div}(\mathbf{v})],\nabla \lambda _{\epsilon }\right) _{\mathcal{O}}.
\label{bvp5.51}
\end{eqnarray}%
A further integration by parts then yields%
\begin{equation}
\left( \lambda _{\epsilon }\func{div}(\mathbf{v}),\Delta \lambda _{\epsilon
}\right) _{\mathcal{O}}=-\int_{\mathcal{O}}\left\vert \nabla \lambda
_{\epsilon }\right\vert ^{2}\func{div}(\mathbf{v})d\mathcal{O-}\frac{1}{2}%
\int_{\mathcal{O}}\lambda _{\epsilon }^{2}\Delta \func{div}(\mathbf{v})d%
\mathcal{O}.  \label{bvp5.53}
\end{equation}

\medskip

\noindent Applying the relations (\ref{bvp4.5}) and (\ref{bvp5.53}) to (\ref{bvp3}),
we then have 
\begin{equation}
\epsilon \left\Vert \Delta \lambda _{\epsilon }\right\Vert _{\mathcal{O}%
}^{2}+\kappa \left\Vert \nabla \lambda _{\epsilon }\right\Vert _{\mathcal{O}%
}^{2}\mathcal{+}\int_{\mathcal{O}}\left( \nabla \mathbf{v}\nabla \lambda
_{\epsilon }\right) \cdot \nabla \lambda _{\epsilon }d\mathcal{O}+\frac{1}{4}%
\int_{\mathcal{O}}\lambda _{\epsilon }^{2}\Delta \func{div}(\mathbf{v})d%
\mathcal{O=-}\int_{\mathcal{O}}G\Delta \lambda _{\epsilon }d\mathcal{O}\text{%
.}  \label{bvp3.5}
\end{equation}

\medskip

Now, concerning the fourth term on left hand side:%
\begin{equation*}
\left\vert \int_{\mathcal{O}}\lambda _{\epsilon }^{2}\Delta \func{div}(%
\mathbf{v})d\mathcal{O}\right\vert \leq \left\Vert \lambda _{\epsilon
}^{2}\right\Vert _{\mathcal{O}}\left\Vert \Delta \func{div}(\mathbf{v}%
)\right\Vert _{\mathcal{O}};
\end{equation*}%
and subsequently applying H\"{o}lder's inequality with conjugates $p=3/2$
and $p^{\ast }=3$, we have then%
\begin{eqnarray}
\left\vert \int_{\mathcal{O}}\lambda _{\epsilon }^{2}\Delta \func{div}(%
\mathbf{v})d\mathcal{O}\right\vert &\leq &meas(\mathcal{O})^{\frac{1}{6}%
}\left\Vert \lambda _{\epsilon }\right\Vert _{L^{6}(\mathcal{O)}%
}^{2}\left\Vert \Delta \func{div}(\mathbf{v})\right\Vert _{\mathcal{O}} 
\notag \\
&\leq &C_{S}meas(\mathcal{O})^{\frac{1}{6}}\left\Vert \nabla \lambda
_{\epsilon }\right\Vert _{\mathcal{O}}^{2}\left\Vert \Delta \func{div}(%
\mathbf{v})\right\Vert _{\mathcal{O}},  \label{bvp3.6}
\end{eqnarray}%
where positive constant $C_{S}$ is that in (\ref{sob}).

\bigskip

Estimating the relation (\ref{bvp3.5}), by means of (\ref{bvp3.6}), we then
obtain 
\begin{equation*}
\epsilon \left\Vert \Delta \lambda _{\epsilon }\right\Vert _{\mathcal{O}%
}^{2}+\left( \kappa -\left[ \frac{C_{S}}{4}meas(\mathcal{O})^{\frac{1}{6}%
}\left\Vert \Delta \func{div}(\mathbf{v})\right\Vert _{\mathcal{O}%
}+\left\Vert \nabla \mathbf{v}\right\Vert _{L^{\infty }(\mathcal{O})}\right]
\right) \left\Vert \nabla \lambda _{\epsilon }\right\Vert _{\mathcal{O}%
}^{2}\leq \left\Vert \nabla G\right\Vert _{\mathcal{O}}\left\Vert \nabla
\lambda _{\epsilon }\right\Vert _{\mathcal{O}};
\end{equation*}%
and so after using assumption (\ref{v_a})(ii), we arrive at%
\begin{equation}
\epsilon \left\Vert \Delta \lambda _{\epsilon }\right\Vert _{\mathcal{O}%
}^{2}+\frac{k}{2}\left\Vert \nabla \lambda _{\epsilon }\right\Vert _{%
\mathcal{O}}^{2}\leq \left\Vert \nabla G\right\Vert _{\mathcal{O}}\left\Vert
\nabla \lambda _{\epsilon }\right\Vert _{\mathcal{O}}.  \label{bvp5}
\end{equation}%
From this estimate, we obtain 
\begin{equation}
\left\Vert \nabla \lambda _{\epsilon }\right\Vert _{\mathcal{O}}\leq \frac{2%
}{k}\left\Vert \nabla G\right\Vert _{\mathcal{O}}.  \label{bvp6}
\end{equation}%
In turn, applying this uniform bound to (\ref{bvp5}), we have also%
\begin{equation*}
\epsilon \left\Vert \Delta \lambda _{\epsilon }\right\Vert _{\mathcal{O}%
}\leq \sqrt{\frac{2\epsilon }{k}}\left\Vert \nabla G\right\Vert _{\mathcal{O}%
}.
\end{equation*}

\smallskip

Consequently, there exists a subsequence $\left\{ \lambda _{\epsilon
}\right\} $ and function $\lambda \in H_{0}^{1}(\mathcal{O})$ such that 
\begin{equation}
\begin{array}{l}
\text{(i) }\displaystyle \lim_{\epsilon \rightarrow 0}\lambda _{\epsilon }=\lambda \text{ 
\emph{weakly} in }H_{0}^{1}(\mathcal{O})\text{ and \emph{strongly} in }L^{2}(%
\mathcal{O}); \\ 
\text{(ii) }\displaystyle \lim_{\epsilon \rightarrow 0}\epsilon \Delta \lambda _{\epsilon
}=0%
\end{array}
\label{converge}
\end{equation}
(see e.g., Theorem 3.27(ii), p.87, of \cite{Mc}).

\smallskip

With the convergences above in hand, we multiply the (\ref{bvp}) by test
function $\varphi \in H^{1}(\mathcal{O})$, and integrate. (Recall that each $%
\lambda _{\epsilon }$ is a strong solution of (\ref{bvp}).) This gives the
relation 
\begin{equation}
\epsilon \int_{\mathcal{O}}\Delta \lambda _{\epsilon }\varphi d\mathcal{O}%
+k\int_{\mathcal{O}}\lambda _{\epsilon }\varphi d\mathcal{O-}\frac{1}{2}%
\int_{\mathcal{O}}\lambda _{\epsilon }\func{div}(\mathbf{v})\varphi d%
\mathcal{O-}\int_{\mathcal{O}}\lambda _{\epsilon }\mathbf{v}\cdot \nabla
\varphi d\mathcal{O}=\int_{\mathcal{O}}G\varphi d\mathcal{O}\text{ \ for
every }\varphi \in H^{1}(\mathcal{O}).  \label{var2}
\end{equation}%
Passing to the limit on left hand side, and invoking (\ref{converge}) we
have that strong $L^{2}$-limit $\lambda $ satisfies (\ref{var}).

\medskip

Moreover, from the relations (\ref{bvp}) and (\ref{bvp2}) for the
regularized problem, and the strong limit posted in (\ref{converge}), we
have the estimate%
\begin{equation}
k\left\Vert \lambda \right\Vert _{\mathcal{O}}\leq \left\Vert G\right\Vert _{%
\mathcal{O}}.  \label{est_f}
\end{equation}

\smallskip

In sum: we have justified the existence of a operator $\mathcal{L}_{0}$,
say, which satisfies, for given $G\in H_{0}^{1}(\mathcal{O})$, $\mathcal{L}%
_{0}(G)=\lambda $, where $\lambda \in L^{2}(\mathcal{O})$ solves (\ref{var}%
), and which yields the estimate 
\begin{equation}
\left\Vert \mathcal{L}_{0}G\right\Vert _{\mathcal{O}}\leq \frac{1}{k}%
\left\Vert G\right\Vert _{\mathcal{O}},  \label{inher}
\end{equation}%
after using (\ref{est_f}). An extension by continuity now yields the
solvability of equation (\ref{linear}) for given $L^{2}$-data $G$, and this
solvability is unique because of the inherent dissipativity in this
equation. Lastly, the estimate (\ref{est}) is just (\ref{inher}).\end{proof}


\begin{thebibliography}{99}
\bibitem{spectral} Aoyama, R. and Kagei, Y., 2016. Spectral properties of
the semigroup for the linearized compressible Navier-Stokes equation around
a parallel flow in a cylindrical domain. \emph{Advances in Differential
Equations}, 21(3/4), pp.265--300.

\bibitem{aubin} Aubin, J.P., 2011. \emph{Applied functional analysis} (Vol.
47). John Wiley \& Sons.

\bibitem{george1} Avalos, G. and Bucci, F., 2014. Exponential decay
properties of a mathematical model for a certain fluid-structure
interaction. In \emph{New Prospects in Direct, Inverse and Control Problems
for Evolution Equations} (pp.49--78). Springer International Publishing.

\bibitem{george2} Avalos, G. and Bucci, F., 2015. Rational rates of uniform
decay for strong solutions to a fluid-structure PDE system. \emph{Journal of
Differential Equations}, 258(12), pp.4398--4423.

\bibitem{clark} Avalos, G. and Clark, T., 2014. A Mixed Variational
Formulation for the Wellposedness and Numerical Approximation of a PDE Model
Arising in a 3-D Fluid-Structure Interaction, \emph{Evolution Equations and
Control Theory}, 3(4), pp.557--578.

\bibitem{dvorak} Avalos, G. and Dvorak, M., 2008. A New Maximality Argument
for a Coupled Fluid-Structure Interaction, with Implications for a
Divergence Free Finite Element Method, \emph{Applicationes Mathematicae},
35(3), pp.259--280.

\bibitem{preprint} Avalos, G. and Geredeli, P.G., 2017. Spectral analysis
and uniform decay rates for a compressible flow-structure PDE model, \emph{%
preprint}.

\bibitem{T1} Avalos G. and Triggiani R., 2007. The Coupled PDE System
Arising in Fluid-Structure Interaction, Part I: Explicit Semigroup Generator
and its Spectral Properties, \emph{Contemporary Mathematics}, {440},
pp.15--54.

\bibitem{T2} Avalos G. and Triggiani R., 2009. Semigroup Wellposedness in
The Energy Space of a Parabolic-Hyperbolic Coupled Stokes-Lam\'{e} PDE of
Fluid-Structure Interactions, \emph{Discrete and Continuous Dynamical Systems%
}, 2(3), pp.417--447.

\bibitem{BA62} Bisplinghoff, R.L. and Ashley, H., 2013. \emph{Principles of
aeroelasticity}. Courier Corporation.

\bibitem{bolotin} Bolotin, V.V., 1963. \emph{Nonconservative problems of the
theory of elastic stability}. Macmillan.

\bibitem{buffa} Buffa, A. and Geymonat, G., 2001. On traces of functions in $%
W^{2,p}(\Omega )$ for Lipschitz domains in R3. \emph{Comptes Rendus de l'Acad%
\'{e}mie des Sciences-Series I-Mathematics}, 332(8), pp.699--704.

\bibitem{buffa2} Buffa, A., Costabel, M. and Sheen, D., 2002. On traces for $%
\mathbf{H}(\func{curl},\Omega )$ in Lipschitz domains. \emph{Journal of
Mathematical Analysis and Applications}, 276(2), pp.845--867.

\bibitem{chen} Chen, G., 1979. Energy decay estimates and exact
boundary-value controllability for the wave-equation in a bounded domain. 
\emph{Journal de Math\'{e}matiques Pures et Appliqu\'{e}es}, 58(3),
pp.249--273.

\bibitem{chorin-marsden} Chorin, A.J. and Marsden, J.E., 1990. \emph{A
mathematical introduction to fluid mechanics} (Vol. 3). New York: Springer.

\bibitem{Chueshov} Chueshov, I., 1999, Introduction to the Theory of
Infinite-Dimensional Dissipative Systems. Acta, Kharkov, (in Russian);
English translation: 2002, Acta, Kharkov.

\bibitem{Igor-note} Chueshov, I., 2013, Personal communication.

\bibitem{Chu2013-comp} Chueshov, I., 2014. Dynamics of a nonlinear elastic
plate interacting with a linearized compressible viscous fluid. \emph{%
Nonlinear Analysis: Theory, Methods \& Applications}, 95, pp.650--665.

\bibitem{Chu2013-inviscid} Chueshov, I., 2014. Interaction of an elastic
plate with a linearized inviscid incompressible fluid. \emph{Communications
on Pure \& Applied Analysis}, 13(5), pp.1459--1778.

\bibitem{newigor} Chueshov, I., 2015. \emph{Dynamics of Quasi-Stable
Dissipative Systems}. New York: Springer.

\bibitem{tube} Chueshov, I. and Fastovska, T., 2016. On interaction of
circular cylindrical shells with a Poiseuille type flow. \emph{Evolution
Equations \& Control Theory}, 5(4), pp.605--629.

\bibitem{springer} Chueshov I. and Lasiecka I., 2010. \emph{Von Karman
Evolution Equations}. {Springer-Verlag}.

\bibitem{supersonic} Chueshov, I., Lasiecka, I. and Webster, J.T., 2013.
Evolution semigroups in supersonic flow-plate interactions. \emph{Journal of
Differential Equations}, 254(4), pp.1741--1773.

\bibitem{delay} Chueshov, I., Lasiecka, I. and Webster, J.T., 2014.
Attractors for Delayed, Nonrotational von Karman Plates with Applications to
Flow-Structure Interactions Without any Damping. \emph{Communications in
Partial Differential Equations}, 39(11), pp.1965--1997.

\bibitem{dcds} Chueshov, I., Lasiecka, I. and Webster, J.T., 2014.
Flow-plate interactions: Well-posedness and long-time behavior. \emph{%
Discrete \& Continuous Dynamical Systems-Series S}, 7(5), pp.925--965.

\bibitem{cr-full-karman} Chueshov, I. and Ryzhkova, I., 2013. Unsteady
interaction of a viscous fluid with an elastic shell modeled by full von
Karman equations. \emph{Journal of Differential Equations}, 254(4),
pp.1833--1862.

\bibitem{ChuRyz2012-pois} Chueshov, I. and Ryzhkova, I., 2013. On the
interaction of an elastic wall with a poiseuille-type flow. \emph{Ukrainian
Mathematical Journal}, 65(1), pp.158--177.

\bibitem{ChuRyz2011} Chueshov, I. and Ryzhkova, I., 2013. A global attractor
for a fluid-plate interaction model. \emph{Communications on Pure \& Applied
Analysis}, 12(4), pp.1635--1656.

\bibitem{berlin11} Chueshov, I. and Ryzhkova, I., 2011, September.
Well-posedness and long time behavior for a class of fluid-plate interaction
models. In \emph{IFIP Conference on System Modeling and Optimization} (pp.
328--337). Springer Berlin Heidelberg.

\bibitem{dV} da Veiga, H.B., 1985. \emph{Stationary Motions and
Incompressible Limit for Compressible Viscous Fluids}, \emph{Houston Journal
of Mathematics}, 13(4), pp.527--544.

\bibitem{dowell1} E. Dowell, 2004. \emph{A Modern Course in Aeroelasticity}. 
{Kluwer Academic Publishers}.

\bibitem{gw} Geredeli, P.G. and Webster, J.T., 2016. Qualitative results on
the dynamics of a Berger plate with nonlinear boundary damping. \emph{%
Nonlinear Analysis: Real World Applications}, 31, pp.227--256.

\bibitem{grisvard} Grisvard, P., 2011. \emph{Elliptic problems in nonsmooth
domains}. Society for Industrial and Applied Mathematics.

\bibitem{kato} Kato, T., 2013. \emph{Perturbation theory for linear operators%
} (Vol. 132). Springer Science \& Business Media.

\bibitem{kesavan} Kesavan, S., 1989. \emph{Topics in functional analysis and
applications}.

\bibitem{conequil2} Lasiecka, I. and Webster, J.T., 2016. Feedback
stabilization of a fluttering panel in an inviscid subsonic potential flow. 
\emph{SIAM Journal on Mathematical Analysis}, 48(3), pp.1848--1891.


\bibitem{LaxPhil} Lax, P.D. and Phillips, R.S., 1960. Local boundary
conditions for dissipative symmetric linear differential operators. \emph{\
Communications on Pure and Applied Mathematics}, 13(3), pp.427--455.

\bibitem{Mc} McLean, W.C.H., 2000. \emph{Strongly elliptic systems and
boundary integral equations}. Cambridge university press.

\bibitem{morawetz} Morawetz, C.S., 1966. Energy identities for the wave
equation, NYU Courant Institute, \emph{Math. Sci. Res. Rep. No.} IMM 346.

\bibitem{necas} Ne\v{c}as, 2012. Direct Methods in the Theory of Elliptic
Equations (translated by Gerard Tronel and Alois Kufner), Springer, New York.

\bibitem{pazy} Pazy, A., 2012. \emph{Semigroups of linear operators and
applications to partial differential equations} (Vol. 44). Springer Science
\& Business Media.


\bibitem{webster} Webster, J.T., 2011. Weak and strong solutions of a
nonlinear subsonic flow-structure interaction: Semigroup approach. \emph{%
Nonlinear Analysis: Theory, Methods \& Applications}, 74(10), pp.3123--3136.

\bibitem{trigg} Triggiani, R., 1989. Wave equation on a bounded domain with
boundary dissipation: an operator approach. \emph{Journal of Mathematical
Analysis and applications}, 137(2), pp.438--461.

\bibitem{valli} Valli, A., 1987. On the existence of stationary solutions to
compressible Navier-Stokes equations. In \emph{Annales de l'IHP Analyse non
lin\'{e}aire} (Vol. 4, No. 1, pp.99--113).
\end{thebibliography}
\end{document}